\documentclass[11pt]{amsart}
\usepackage{bbm}
\usepackage{amsfonts}
\usepackage{amsmath}
\allowdisplaybreaks
\usepackage{amssymb}
\usepackage{fancyhdr}
\usepackage{setspace}
\usepackage{latexsym}
\usepackage{mathrsfs}
\usepackage{amsthm}

\usepackage{cite,enumitem,graphicx}
\usepackage[backref,colorlinks,linkcolor=red,anchorcolor=green,citecolor=blue]{hyperref}
\usepackage{cleveref}
\usepackage[english]{babel}

\numberwithin{equation}{section}
\newtheorem{theorem}{Theorem}[section]
\newtheorem{proposition}[theorem]{Proposition}
\newtheorem{lemma}[theorem]{Lemma}
\newtheorem{corollary}[theorem]{Corollary}
\newtheorem{definition}[theorem]{Definition}
\newtheorem{remark}[theorem]{Remark}
\newtheorem{example}[theorem]{Example}

\newcommand{\R}{\mathbb R}

\newcommand{\bt}{\begin{theorem}}
\newcommand{\et}{\end{theorem}}
\newcommand{\bl}{\begin{lemma}}
\newcommand{\el}{\end{lemma}}
\newcommand{\bd}{\begin{definition}}
\newcommand{\ed}{\end{definition}}
\newcommand{\bc}{\begin{corollary}}
\newcommand{\ec}{\end{corollary}}
\newcommand{\bp}{\begin{proof}}
\newcommand{\ep}{\end{proof}}
\newcommand{\bx}{\begin{example}}
\newcommand{\ex}{\end{example}}
\newcommand{\bi}{\begin{exercise}}
\newcommand{\ei}{\end{exercise}}
\newcommand{\bo}{\begin{proposition}}
\newcommand{\eo}{\end{proposition}}
\newcommand{\br}{\begin{remark}}
\newcommand{\er}{\end{remark}}
\newcommand{\be}{\begin{equation}}
\newcommand{\ee}{\end{equation}}
\newcommand{\ba}{\begin{align}}
\newcommand{\ea}{\end{align}}
\newcommand{\bn}{\begin{enumerate}}
\newcommand{\en}{\end{enumerate}}
\newcommand{\bg}{\begin{align*}}
\newcommand{\bcs}{\begin{cases}}
\newcommand{\ecs}{\end{cases}}

\newcommand{\bean}{\begin{eqnarray*}}
\newcommand{\eean}{\end{eqnarray*}}
\newcommand{\loc}{\operatorname{\rm loc}}

\begin{document}

\title[Comparison principles for subquadratic fractional $p$-Laplacian equation]{Comparison principles and symmetry for subquadratic fractional $p$-Laplacian equations}
\date{}

\author{Dong Ye}
\address{School of Mathematical Sciences,  Key Laboratory of MEA (Ministry of Education) \& Shanghai Key Laboratory of PMMP,  East China Normal University, Shanghai 200241, China}
\email{dye@math.ecnu.edu.cn}
\author{Weimin Zhang}
\address{School of Mathematical Sciences, Zhejiang Normal University, Jinhua 321004, China}
\email{zhangweimin2021@gmail.com}
\maketitle

\begin{abstract}
 This paper establishes a new comparison principle framework for the subquadratic fractional $p$-Laplacian, i.e.~$1 < p < 2$
 under minimal regularity assumptions, that has remained a significant challenging issue due to the singularity of the operator. Our results provide the essential analytical tools required for the moving plane method in this setting. We prove first a weak comparison principle for $(-\Delta)_p^s u = f(u)$ in bounded domains with sufficiently small measure, where only the boundedness of the weak solution is required. More importantly, we establish a strong comparison principle for continuous weak solutions in the parameter range $s \in (0, \frac{1}{2})$ and $\frac{1}{1-s} < p < 2$. Our proof introduces a localized barrier function and does not require any H\"{o}lder regularity of the weak solution, nor any smoothness of the domain. This presents a substantial contrast over previous study, which relied heavily on H\"older or even $C^{1,1}$ regularity.
As a direct application, we employ these comparison principles to prove the symmetry of weak solutions to $(-\Delta)_p^s u = f(u)$ under mild assumptions, which significantly extend existing symmetry theories for nonlocal quasilinear equations.
 \end{abstract}



\section{Introduction}\label{s1}

The qualitative analysis of solutions to nonlinear elliptic equations, particularly their geometric properties such as radial symmetry and monotonicity, constitutes a central theme in modern partial differential equations. A highly elegant and robust mathematical framework to address these questions is the moving-plane method, initiated by Alexandrov \cite{A_AMPA1961} and further developed by Serrin \cite{S_ARMA1971} in classical potential theory. For semilinear elliptic equations, Gidas-Ni-Nirenberg \cite{GNN_CMP1979} established symmetry properties in bounded smooth domains by developing maximum principles in narrow domains. Berestycki-Nirenberg \cite[Proposition 1.1]{BN_BSBM1991} improved the  maximum principle from narrow domains to more general domains with small measure, allowing the moving-plane procedure with weaker regularity assumptions on the domain and the solution. 

\smallskip
When shifting to the quasilinear regime governed by the $p$-Laplacian  $\Delta_p u := \text{div}(|\nabla u|^{p-2}\nabla u)$, $p\ne 2$, the structural nonlinearity of the operator introduces serious analytical difficulties due to its degenerate ($p>2$) or singular ($1<p<2$) behavior. In the absence of linearity, the classical maximum principles should be replaced by weak or/and strong comparison principles.  The strong comparison principle is a very hard issue in the quasilinear case. For example, let $u, v$ be ordered $p$-harmonic functions, i.e. $\Delta_p u = \Delta_pv = 0$ and $u \ge v$ in domain $\Omega \subset \R^N$, we wonder if the alternative $u > v$ in $\Omega$ and $u\equiv v$ holds true. The only general affirmative answer was given by Manfredi \cite{M_PAMS88} in the plane $\R^2$, and the situations for $p \ne 2$,  $N\ge 3$ remain a famous longstanding open question.

The researchers seek then the strong comparison principle for $p$-Laplacian problem under suitable assumptions. Consider
\begin{equation*}
-\Delta_{p}u-\Lambda u\ge-\Delta_{p}v-\Lambda v\quad \mbox{and}\quad u \ge v\quad \text{in } U, 
\end{equation*}
for some constant $\Lambda \ge 0$, we ask if there holds either $u\equiv v$ or $u > v$ in $U$. Let $u, v \in C^1(\overline{\Omega})$. For $\Lambda = 0$, Tolksdorf \cite{T_CPDE1983} and Guedda-V\'eron \cite[Proposition 2.1]{GV_NA1989} derived the strong comparison principle under the restriction that   $\nabla v$ never vanishes in $U$ or $\{x\in U: u(x)=v(x)\}$ is compact. For general $\Lambda \ge 0$, Damascelli \cite{D_AIHP1998} established a strong comparison principle on each connected component of  $U\backslash Z$, where $Z=\{x\in U: \nabla u(x)= \nabla v(x)= 0\}$. 

Applying the comparison principles, for $p > 2$, Damascelli-Sciunzi \cite{DS_JDE2004, DS_CVPDE2006} considered the quasilinear equation 
\begin{equation}\label{2606180023}
\begin{aligned}
\begin{cases}
-\Delta_p u =f(u) &~\mbox{in}\; \Omega,\\
u=0 &\mbox{on}\; \partial \Omega.\\
\end{cases}
\end{aligned}
\end{equation} 
They proved the symmetry of positive solutions if $f \in C^{0,1}_{loc}$ and $f(t) > 0$ for $t > 0$. It is worthy to point out that there exist non-radial positive solutions on $\Omega=B_1$ if $p>2$ and without sign constraint on $f$, see \cite{B_PIAS2000, KS_N1990}. 

For $1<p<2$, Damascelli-Pacella \cite{D_AIHP1998, DP_ASNSPCS1998}  established symmetry and monotonicity properties of positive solutions for subquadratic $p$-Laplacian equation \eqref{2606180023}, assuming that $f \in C^{0,1}_{loc}$ and $u \in C^1(\overline\Omega)$. 

\smallskip
Recently, there are rapid growth of interests for nonlocal, integro-differential problems, see for example \cite{BT_APDE2025, KKL_JMPA2019, CLL_AM2017, GJS_arXiv2025, GL_MA2024, JW_AMPA2016, RS_CVPDE2014, BLS_AM2018, DN_APDE2025,  IMS_RMI2016} and references therein. A typical model is the fractional $p$-Laplacian operator $(-\Delta)_p^s$ with $s\in (0, 1)$ and $p>1$, whose weak formulation is given in \eqref{2608021653} below. If $u\in C_c^2(\mathbb{R}^N)$, for either $1<p\le \frac{2}{2-s}$ and $ \nabla u(x_0)\neq 0$, or $p>\frac{2}{2-s}$, $(-\Delta)_p^s u$ admits a pointwise principal-value representation
\begin{equation}\label{2608021824}
(-\Delta)_p^s u(x_0)=2\,\underset{\varepsilon\to 0}{\lim}\int_{\mathbb{R}^N\backslash B_{\varepsilon}(x_0)}\frac{|u(x)-u(y)|^{p-2}(u(x)-u(y))}{|x-y|^{N+sp}}{\rm d}y.
\end{equation}
See \cite[Lemma 2.6]{KKL_JMPA2019}. 

\smallskip
To implement the moving plane procedure for non-negative solutions to
\begin{equation}\label{main}
\begin{aligned}
\begin{cases}
(-\Delta)_p^s u= f(u)\quad&\mbox{in}\; \Omega\\
u=0&\mbox{in}\; \mathbb{R}^N\backslash \Omega,\\
\end{cases}
\end{aligned}
\end{equation}
one needs the comparison principles between $u$ and its reflection. Let 
\begin{equation}\label{2412171452}
\begin{cases}
\begin{aligned}
(-\Delta)_p^s u-(-\Delta)_p^s u^*&\ge V(x)(u-u^*) \quad &&\mbox{in}\; U,\\
u&\ge u^* &&\mbox{in}\; \mathbb{R}^N_+\backslash U,\\
\end{aligned}
\end{cases}
\end{equation}
where $V\in L_{loc}^1(U)$, $\mathbb{R}^N_+=\{(x', x_N)\in \mathbb{R}^{N-1}\times \mathbb{R}: x_N>0 \}$ is the upper half-space, $U\subset \mathbb{R}^N_+$ is a bounded open set. Here $u^*(x):=u(x^*)$, and $x^*=(x', -x_N)$ for any $x=(x', x_N)\in \mathbb{R}^N_+$. 

Due to the global coupling of the nonlocal tail, we can imagine that the comparison between $u$ and its reflected profile $u^*$ is highly non-trivial. 

\smallskip 
Consider first $p=2$, i.e. the linear fractional Laplacian $(-\Delta)^s$ situation. Ros-Oton--Serra \cite[Lemma 5.1]{RS_CVPDE2014}, Jarohs-Weth \cite[Proposition 3.5]{JW_AMPA2016} proved the weak comparison principle  for  \eqref{2412171452} in domains with small measure. Furthermore, Jarohs-Weth \cite[Proposition 3.6]{JW_AMPA2016} proved the strong comparison principle for antisymmetric functions, allowing to establish the symmetric property of non-negative solutions to
\eqref{main} with $p=2$, bounded and convex domain $\Omega$ symmetric with respect to some hyperplane, and $f\in C_{\loc}^{0, 1}$. 

On the other hand, Chen-Li-Li \cite{CLL_AM2017} worked  on $C^{1, 1}_{loc}$ solutions $u$ so that $(-\Delta)^s u$ has an integral representation as in \eqref{2608021824}, and they
  can carry out the method of moving plane in bounded or unbounded domains. For further related results on strong comparison principles  for  \eqref{2412171452} with $p=2$, we also refer to \cite{BMS_JAM2018, GP_JAM2025} and references therein.

\medskip
 For the equations \eqref{main} and \eqref{2412171452} with $p \neq 2$, the problem is more subtle since we combine the nonlocal effect with the nonlinear differential operator. Early symmetry and monotonicity results relied heavily on the direct moving-plane method of Chen-Li \cite{CL_AM2018}, and required $C^{1,1}_{\text{loc}}$ regularity of the solution $u$. 
 
 However, it is difficult to expect $C^{1, 1}_{loc}$ regularity for weak solutions to \eqref{main} in general. For instance, we have the H\"older regularity for solutions to fractional $p$-Laplacian equations (see \cite{ BLS_AM2018, DN_APDE2025, GL_MA2024, IMS_RMI2016}); the authors in \cite{LS_ArXiv2026, BL_AM2017, BDLM_JMPA2025, BDLMS_CVPDE2025} showed the corresponding Sobolev regularity results. 
 
 Recently, for $(s, p)$-harmonic functions, Giovagnoli-Jesus-Silvestre \cite{GJS_arXiv2025} show that the solutions of $(-\Delta)_p^s u=0$ are $C^{1, \alpha}$ regular if $2\le p<\frac{2}{1-s}$. Biswas-Topp \cite{BT_APDE2025}
investigate the H\"older regularity of $(-\Delta)_p^su = g\in L^\infty_{loc}$ for $p > 1$ and $s\in (0, 1)$, they prove that $u \in C^{0,\gamma}_{loc}$ with $\gamma = \min\{1, \frac{sp}{p-1}\}$, provided that $\frac{sp}{p-1} \ne 1$. These seem to be the best regularity results available for the fractional $p$-Laplacian. Notice that when $p \ne 2$, due to the limitation from the nonlinear operator, the higher regularity of $g$ beyond $g \in L^\infty$ seems not improve the H\"older exponent of $u$ in the known sharp theory. We remark also that the H\"older continuity could become increasingly weak as $s \to 0$.
 
 It is worth mentioning that even for $p$-harmonic functions, i.e. solutions to $-\Delta_p u =0$, we cannot hope $C^{1, 1}_{loc}$ smoothness generally. Iwaniec-Manfredi \cite{IM_RMI1989} gave the optimal $C^{1, \alpha}$ theory for $p$-harmonic functions in dimension two, with
 \[
\alpha_{\mathrm{opt}}(2,p)
=\frac{1}{6}\left(\frac{p}{p-1}+\sqrt{1+\frac{14}{p-1}+\frac{1}{(p-1)^2}}\right).
\]
Hence when $p>2$, there exist $p$-harmonic functions on the plane, which do not satisfy the $C^{1, 1}_{loc}$ regularity. To the best of our knowledge, the optimal regularity issue for $p$-harmonic functions in $N\ge 3$ remains widely open. 

\medskip
Consequently, it is natural to study solutions to fractional $p$-Laplacian equations with lower regularity assumption. 
 
 \smallskip
For $p>2$, Biswas-Roy-Sen \cite[Theorem 4.4]{BRS_arXiv2026} established the strong comparison principle for solution to 
\begin{equation}
\label{fractionalp1}
\begin{aligned}
\begin{cases}
(-\Delta)_p^s u-(-\Delta)_p^s u^*= f(u)-f(u^*)\quad&\mbox{in}\; U\subset \mathbb{R}^N_+\\
u\ge u^* &\mbox{in}\;  \mathbb{R}^N_+.\\
\end{cases}
\end{aligned}
\end{equation}
Here $u^*$, $U$ are the same as for \eqref{2412171452}, and $f \in C^{0, 1}_{loc}$. Their result can be summarized as follows.
\begin{itemize}
\item If $p>\max\big\{2, \frac{1}{1-s}\big\}$ and  $u$ is locally $\frac{sp}{p-1}+\varepsilon$-H{\"o}lder continuous for some $\varepsilon>0$, then either $u>u^*$ in $U$ or $u\equiv u^*$ in $\mathbb{R}^N$; 
\item If $2<p<\frac{2}{1-s}$ and $u\in C^1(U)$, then either $u>u^*$ in $\{x\in U: \nabla u^*(x)\neq 0\}$ or $u\equiv u^*$ in $\mathbb{R}^N$. 
\end{itemize}
Furthermore, they  proved that any weak solution to \eqref{main} satisfies the above regularity requirement. As application, when $\Omega$ is a strictly convex, bounded domain with $C^{1, 1}$ boundary, and symmetric with respect to some hyperplane, they showed the  symmetry and monotonicity for positive $C^1$ solutions to \eqref{main}. 

\smallskip
Despite the above developments, the strong comparison principle of weak solutions to \eqref{main} in the subquadratic regime $1<p<2$ remained widely open. One main obstacle is that the singularity of the operator for $p \in (1,2)$ makes integral estimates technically challenging.  As far as we are aware, Jarohs \cite{J_ANS2018} showed the strong comparison between super and sub-solutions to a special equation $(-\Delta)_p^s u = q(x)|u|^{p-2}u$ with $\frac{1}{1-s}< p < 2$. Iannizzotto-Mosconi-Papageorgiou \cite[Theorem 2.7]{IMP_MN23} established an interesting strong comparison between positive continuous super and sub-solutions to $(-\Delta)_p^s u = f(u)$ with general $p > 1$ and $f \in C(\R)\cap BV_{loc}(\R)$, however, the condition for the subsolution $v \in W_0^{s, p}(\Omega)$ prevents the application to the reflected part $u^*$.

\medskip
Our aim here is to establish weak and strong comparison principles for weak solutions to \eqref{2412171452} or \eqref{fractionalp1}, in the subquadratic regime (i.e. $p < 2$), under mild conditions on the regularity of the solution $u$, and the domain. 

We first establish the weak comparison principle in open sets with small measures,  it seems to be the first result of this flavor for the subquadratic fractional $p$-Laplacian setting, which is the nonlocal, quasilinear analogue of the classical narrow domain
principle and is an essential tool for the moving plane procedure. Moreover, we require only the boundedness of the weak solution.

Let us define the space for weak solutions as follows. Let $s\in (0, 1)$ and $p>1$,
\begin{equation}\label{2607282240}
\widetilde{W}^{s, p}(\Omega):=\left\{u\in L^p_{loc}(\mathbb{R}^N): \int_{\Omega}\int_{\mathbb{R}^N}\frac{|u(x)-u(y)|^{p}}{|x-y|^{N+sp}} {\rm d}x{\rm d}y<\infty\right\}.
\end{equation}
More precise definition of weak solution to \eqref{main} can be found in Section \ref{S_P}. For $U\subset \mathbb{R}^N_+$, we denote $\mathcal{T}(U)=\{x^*: x\in U\}$. 

\begin{theorem}\label{2508070852}
Let $N\ge 1$, $s\in (0, 1)$, $1<p< 2$, $U\subset \mathbb{R}^N_+$ be a bounded  open set, and $V\in L_{loc}^{1}(U)$ with $V_+(x)=\max\{V(x), 0\}\in L^{\infty}(U)$. Then there exists some $C>0$ depending only on $N, p$ and $s$ such that for any weak solution $u\in \widetilde{W}^{s, p}(U)\cap\widetilde{W}^{s, p}(\mathcal{T}(U))\cap L^{\infty}(\mathbb{R}^N)$ of  \eqref{2412171452} satisfying 
\begin{equation}\label{2603291622}
|U|^{\frac{sp}{N}}\|V_+\|_{L^{\infty}(U)}\|u\|_{L^{\infty}(\mathbb{R}^N)}^{2-p}<C,
\end{equation}
we have $u\ge u^*$ a.e. in $U$. 
\end{theorem}

Notice that if $V\le 0$ in $U$, the condition \eqref{2603291622} imposes no restriction on $|U|$. Furthermore, the above weak comparison holds for the non-autonomous case
\begin{equation*}\label{2412171452a}
\begin{cases}
\begin{aligned}
(-\Delta)_p^s u-(-\Delta)_p^s u^*&\ge \Lambda (u-u^*) + g(u^*, x)-g(u, x)\quad &&\mbox{in}\; U\subset \mathbb{R}^N_+,\\
u&\ge u^* &&\mbox{in}\; \mathbb{R}^N_+\backslash U,\\
\end{aligned}
\end{cases}
\end{equation*}
where $\Lambda\ge 0$, $g(t, x)$ is nondecreasing in $t$, and $u\in L^{\infty}(\mathbb{R}^N)$. Then there exists 
$$\delta=\delta(\|u\|_{ L^{\infty}(\mathbb{R}^N)}, \Lambda, N, p, s)>0$$ such that $u\ge u^*$ in $U$ whenever $|U|<\delta$.

This is a similar result as for the $p$-Laplacian case in \cite[Theorem 1.2(c)]{D_AIHP1998} and \cite[Theorem 2.1(b)]{GYZ_CAM2002}. Theorem \ref{2508070852} extends also the result for linear fractional Laplacian case (i.e.~$p=2$) in \cite[Lemma 5.1]{RS_CVPDE2014} to the fractional quasilinear case $1<p< 2$.

\begin{remark}
Theorem \ref{2508070852} allows to start the moving plane procedure, if the sectional curvatures are positive on $\partial\Omega$. As mentioned in \cite{dLN_JMPA1982}, the above arguments also work in general convex domains. Accordingly, for locally Lipschitz function $f$, the bounded positive solutions of  \eqref{main} are monotonic along some directions near the boundary. Comparing with the case $p>2$ in \cite{BRS_arXiv2026}, we need not to impose any smoothness for $\partial\Omega$.
\end{remark}

Using some localized barrier functions, we establish the following strong comparison principle with a more restricted parameter regime.

\begin{theorem}\label{2605211200}
Let $N\ge 1$,  $s\in (0, \frac12)$, $\frac{1}{1-s}<p< 2$, $U\subset \mathbb{R}^N_+$ be an  open set, and $V\in L_{loc}^{\infty}(U)$. If $u\in \widetilde{W}^{s, p}(U)\cap\widetilde{W}^{s, p}(\mathcal{T}(U))\cap L^{\infty}(\mathbb{R}^N)\cap C(\mathbb{R}^N)$ is a  solution of   \eqref{2412171452} and $u\ge u^*$ in $\mathbb{R}^N_+$, there holds then either $u\equiv u^*$ in $\mathbb{R}^N$ or $u-u^*>0$ in $U$.
\end{theorem}


The proof introduces a localized barrier function and applies the weak comparison principle Theorem \ref{2508070852}. This is a fundamentally different approach compared to the measure-theoretic and Ishii-Lions methods used for the superquadratic case ($p > 2$). Crucially, our method does not require any regularity assumptions (such as H\"older continuity) on the solution nor on the domain.

Comparing to the case $p>2$, the method of \cite{BRS_arXiv2026} relies heavily on the $\frac{sp}{p-1}+\varepsilon$-H{\"o}lder regularity of solutions for $p>\max\big\{2, \frac{1}{1-s}\big\}$ and the $C^1$ regularity of solutions for $2<p<\frac{2}{1-s}$. As already mentioned, by the recent literature \cite{ BT_APDE2025, BDLM_JMPA2025, BDLMS_CVPDE2025, GL_MA2024}, the best regularity for fractional $p$-Laplacian equation with the parameters range $\frac{1}{1-s}<p<2$ only reaches $C^\alpha_{loc}$ with $\alpha = \frac{sp}{p-1} < 1$, which goes to $0$ as $s \to 0$. The value $\frac{1}{1-s}$ may be a borderline between H\"older regularity and Lipschitz regularity. This is the reason why we cannot derive strong comparison principle as in \cite[Theorem 4.4(ii)]{BRS_arXiv2026}. Moreover, our result extends \cite[Theorem 2]{CL_AM2018} under much weaker regularity condition. 

\begin{remark}
The restriction $s\in(0,\frac12)$ and $p > \frac{1}{1-s}$ is necessary for our approach, in order to construct a non-trivial, non-negative barrier function $g \in C^2_c(\mathbb{R}^N)$ such that 
$$\int_{\mathbb{R}^{N}}\frac{|g(x)-g(y)|^{p-1}}{|x-y|^{N+sp}}{\rm d}y$$
 is locally bounded. 
See Lemma \ref{2605121130} below, i.e. \cite[Lemma 3.5(i)]{J_ANS2018}. 
\end{remark}

As application of the above weak and strong comparison principles, we prove some symmetry results for positive weak solution to \eqref{main}.  Remark that the solutions only need to be continuous, and the domain only needs to be bounded and convex, without further boundary smoothness requirement.

\begin{theorem}\label{2602111215}
Let $N\ge 1$, $s\in (0, \frac12)$, $\frac{1}{1-s}<p< 2$, and let $\Omega\subset\mathbb{R}^N$ be a bounded open set, convex in $x_N$ direction and symmetric with respect to the hyperplane $\{x_N=0\}$. Assume that $f \in C^{0,1}_{loc}$ and $u\in W_0^{s, p}(\Omega)\cap C(\mathbb{R}^N)$ is a non-negative weak solution to \eqref{main}, then  $u$ is symmetric with respect to the hyperplane $\{x_N=0\}$. 
\end{theorem}

The same argument also applies to non-autonomous nonlinearities $g(u,x)$, provided that $g$ is locally Lipschitz continuous in $u$, uniformly in $x$, and satisfies the appropriate symmetry and monotonicity assumptions. 

When $f$ has Sobolev subcritical or critical growth, any weak solution of \eqref{main} is bounded and continuous, which yields the following.
\begin{corollary}
\label{2603271652}
Let $N\ge 1$, $s\in (0, \frac12)$, $\frac{1}{1-s}<p< 2$, and $\Omega=B_R$, a Euclidean ball centered at the origin. Suppose that $f\in C_{loc}^{0, 1}(\R)$ satisfies
\begin{equation}\label{2608241033}
|f(u)|\le C\left(1+|u|^{\frac{Np}{N-sp}-1}\right).
\end{equation}
If $u\in W_0^{s,p}(B_R)$ is a nontrivial non-negative weak solution of \eqref{main}, then  $u$ is positive in $B_R$, radially symmetric and strictly decreasing with respect to the radius.
\end{corollary}

The paper is organized as follows. We will introduce some necessary notations, functional setting and preliminary results in Section \ref{S_P}. Section \ref{S2} is devoted to the proof of the weak comparison principle, Theorem \ref{2508070852}. Section \ref{2606091353}  describes the barrier construction and establishes the strong comparison principle, Theorem \ref{2605211200}. Finally, in Section \ref{2606091358}, we apply these results to prove the symmetry result, Theorem \ref{2602111215}.

\section{Preliminaries }\label{S_P}
\subsection{Notations and definitions}
Throughout this paper, we use the following notations.
\begin{itemize}
\item $\mathcal{P}(t)=|t|^{p-2}t$ for any $t\in \mathbb{R}$.
\item $\omega_{N}$ denotes the area of the unit sphere in $\mathbb R^{N+1}$, especially $\omega_{0}=2$.
\item ${\mathbbm 1}_A$ is the characteristic function of the set $A$.
\item $\mathbb{R}^N_+=\{x=(x', x_N)\in \mathbb{R}^{N-1}\times \mathbb{R}: x_N> 0\}$, $\mathbb{R}^N_-=\{x=(x', x_N)\in \mathbb{R}^{N-1}\times \mathbb{R}: x_N< 0\}$.
\item $x^*=(x', -x_N)$ for $x=(x', x_N)\in\mathbb{R}^{N-1}\times \mathbb{R}$, and $\mathcal{T}(U)=\{x^*: x\in U\}$ for any $U\subset \mathbb{R}^N$.
\item For $u: \mathbb{R}^N\to \mathbb{R}$, we denote $u^*(x)=u(x^*)$ and $J_u(x,y)=|u(x)-u(y)|^{p-2}(u(x)-u(y))$. 
\item $w_{\pm}(x)=\max\{0, \pm w(x)\}$.
\item $\|\cdot\|_{L^p(A)}$ stands for the usual norm of $L^{p}(A)$.
\item $B_R(x_0)$ means the ball of radius $R$ in $\mathbb R^N$, centered at $x_0$. We denote by $B_R$ if $x_0 = 0$.
\end{itemize}

For $s\in (0, 1)$ and $p>1$, we denote the Gagliardo seminorm by
\begin{equation}\label{2608311117}
[u]_{s,p}:=\left( \int_{\mathbb{R}^N}\int_{\mathbb{R}^N}{\frac{|u(x)-u(y)|^p}{|x-y|^{N+sp}}}{\rm d}x{\rm d}y\right)^{1/p}.
\end{equation}
Let
\[
W^{s,p}(\mathbb{R}^N):=\{u\in L^p(\mathbb{R}^N): [u]_{s,p}<\infty\}
\]
be endowed with the norm
\[
\|u\|_{W^{s,p}(\mathbb{R}^N)}:=\|u\|_{L^p(\mathbb{R}^N)}+[u]_{s,p}.
\]
For any open set $\Omega\subset \mathbb{R}^N$, we denote the subspace
\begin{equation}\label{2603311321}
W_0^{s,p}(\Omega):=\left\{u\in W^{s,p}(\mathbb{R}^N):  u=0 \; \text{~in}~\mathbb{R}^N \backslash \Omega\right\},
\end{equation}
which can be equivalently renormed with $\|u\|=[u]_{s,p}$ if $\Omega$ is  bounded, see \cite[Theorem 7.1]{DPV_BSM12}. It is well-known that $W_0^{s,p}(\Omega)$ is a uniformly convex Banach space for any $1 < p < \infty$. 

\begin{remark}
As we work with weak solutions, the equalities and inequalities hold in the weak sense, or almost everywhere. We often omit writing $a.e.$ for simplicity. For example, in \eqref{main} or in the definition \eqref{2603311321}, we mean $u=0$ in $\mathbb{R}^N \backslash \Omega$ almost everywhere.
\end{remark}

Recall that (see \eqref{2607282240})
\begin{equation*}
\widetilde{W}^{s, p}(\Omega):=\left\{u\in L^p_{loc}(\mathbb{R}^N): \int_{\Omega}\int_{\mathbb{R}^N}\frac{|u(x)-u(y)|^{p}}{|x-y|^{N+sp}} {\rm d}x{\rm d}y<\infty\right\}.
\end{equation*}
Then we define $(-\Delta)_p^s$ as the operator from $\widetilde{W}^{s, p}(\Omega)$ to the dual space $W_0^{s, p}(\Omega)'$ of $W_0^{s,p}(\Omega)$, i.e.
\begin{equation}\label{2608021653}
\big\langle (-\Delta)_p^s u, \varphi \big\rangle=\int_{\mathbb{R}^{N}}\int_{\mathbb{R}^{N}}{\frac{J_u(x,y)(\varphi(x)-\varphi(y))}{|x-y|^{N+sp}}}{\rm d}x{\rm d}y\;\;\quad \forall\, u\in \widetilde{W}^{s, p}(\Omega),\;\; \varphi\in W_0^{s,p}(\Omega),
\end{equation}
where $J_u(x,y)=|u(x)-u(y)|^{p-2}(u(x)-u(y))$. Now we can state the definition of weak solutions to \eqref{main} and \eqref{2412171452}.

\begin{definition}
We call $u\in W_0^{s,p}(\Omega)$ a weak solution to \eqref{main} if $f(u)\in  W_0^{s, p}(\Omega)'$ and
\begin{equation*}
\int_{\mathbb{R}^{N}}\int_{\mathbb{R}^{N}}{\frac{J_u(x,y)(\varphi(x)-\varphi(y))}{|x-y|^{N+sp}}}{\rm d}x{\rm d}y= \big\langle f(u), \varphi \big\rangle, \quad \forall\; \varphi \in W_0^{s,p}(\Omega).
\end{equation*} 
\end{definition}

\begin{definition}
We say that $u\in \widetilde{W}^{s, p}(U)\cap \widetilde{W}^{s, p}(\mathcal{T}(U))$ satisfies \eqref{2412171452}, if $u\ge u^*$ in $\mathbb{R}^N_+\backslash U$, $V(x)(u-u^*)\in W_0^{s,p}(U)'$, and there holds
\[
\big\langle (-\Delta)_p^s u, \varphi \big\rangle-\big\langle (-\Delta)_p^s u^*, \varphi \big\rangle \ge \int_{U}V(x)(u-u^*)\varphi {\rm d}x\quad\mbox{ for all  $\varphi\in W_0^{s,p}(U)$ with $\varphi\ge 0$}.
\]
\end{definition}

\subsection{Some preliminary lemmas}

 The first result shows the relationship between $\widetilde{W}^{s, p}(\Omega)$ and $W_0^{s, p}(\Omega)$. For more properties about $\widetilde{W}^{s, p}(\Omega)$, we refer to \cite[Section 2.2]{J_ANS2018}.

\begin{lemma}[\!\!\cite{J_ANS2018}, Lemma 2.6]\label{2605261121}
Let $\Omega$ be a bounded open set and $u\in \widetilde{W}^{s, p}(\Omega)$ with $u=0$ in $\mathbb{R}^N\backslash \Omega$. Then $u\in W_0^{s, p}(\Omega)$.
\end{lemma}
 
For any open set $\Omega\subset\mathbb{R}^N$, we denote the first eigenvalue of $(-\Delta)_p^s$ in  $W_0^{s, p}(\Omega)$ by
\[
\lambda_{1, s, p}(\Omega)=\inf\left\{[u]_{s, p}^p: u\in W_0^{s,p}(\Omega), \; \|u\|_{L^p(\Omega)}=1\right\},
\]
which possesses the following lower bound estimate.
\begin{lemma}\label{2603291635}
Assume that $N\ge 1$, $s\in (0, 1)$, $p>1$ and $\Omega\subset \mathbb{R}^N$ is an open set. Then 
\begin{equation}\label{2607291936}
\lambda_{1, s, p}(\Omega)\ge \frac{2\omega_{N-1}^{1+\frac{sp}{N}}}{spN^{\frac{sp}{N}}}|\Omega|^{-\frac{sp}{N}}.
\end{equation}
\end{lemma}
\begin{proof}
Let $u\in W_0^{s,p}(\Omega)$ and $\|u\|_{L^p(\Omega)}=1$. The estimate is trivial if $|\Omega|=\infty$, let then $|\Omega|<\infty$.
\begin{align*}
[u]_{s, p}^p & = \int_{\Omega\times \Omega}\frac{|u(x)-u(y)|^p}{|x-y|^{N+sp}} {\rm d}x{\rm d}y+ 2\int_{\Omega\times \Omega^c}\frac{|u(x)|^p}{|x-y|^{N+sp}} {\rm d}x{\rm d}y\\
&\ge 2\|u\|_{L^p(\Omega)}^p\inf_{x\in \Omega}\int_{\Omega^c}\frac{1}{|x-y|^{N+sp}} {\rm d}y,
\end{align*}
Hence we have
\begin{align*}
[u]_{s, p}^p \ge 2\int_{B_{r(\Omega)}^c}\frac{1}{|y|^{N+sp}} {\rm d}y =2\frac{\omega_{N-1}}{sp}r(\Omega)^{-sp} =\frac{2\omega_{N-1}^{1+\frac{sp}{N}}}{spN^{\frac{sp}{N}}}|\Omega|^{-\frac{sp}{N}},
\end{align*}
where $B_{r(\Omega)}$ satisfies $|B_{r(\Omega)}|=|\Omega|$.
\end{proof}

The following are some elementary inequalities, where we sometimes omit  the proof.
\begin{lemma}\label{Lo}
For any $1<p\le 2$ and $a > b$, there holds
\[
\int_b^a|t|^{p-2} {\rm d}t \ge (a-b)\min \big\{|a|^{p-2}, |b|^{p-2}\big\}.
\]
\end{lemma}

\begin{lemma}\label{2605111444}
 For any $1<p\le 2$, there hold  
\[
 \left| |t+\delta|^{p-2}(t+\delta)-|t|^{p-2}t\right|\le 2^{2-p}|\delta|^{p-1}\quad \mbox{for all } t, \delta\in \mathbb{R};
\]
and let $M>0$, we have
\[
|t+\delta|^{p-2}(t+\delta)-|t|^{p-2}t\ge (p-1)(M+\delta)^{p-2}\delta \quad \mbox{for all } |t|\le M \mbox{ and } \delta\ge 0.
\]
\end{lemma}
\begin{proof}
Given any $t, \delta\in \mathbb{R}$,
\[
\begin{aligned}
\Big| |t+\delta|^{p-2}(t+\delta)-|t|^{p-2}t\Big|&=(p-1)\Big|\int_{t}^{t+\delta} |\tau|^{p-2} d\tau\Big| \le 2(p-1)\int_{0}^{\frac{|\delta|}2} |\tau|^{p-2} d\tau=2^{2-p}|\delta|^{p-1}.
\end{aligned}
\]
If now $|t|\le M$ and $\delta\ge 0$,
\begin{align*}
|t+\delta|^{p-2}(t+\delta)-|t|^{p-2}t & =(p-1)\int_{t}^{t+\delta} |\tau|^{p-2} d\tau\\
& \ge (p-1)\int_{M}^{M+\delta} |\tau|^{p-2} d\tau\ge (p-1)(M+\delta)^{p-2}\delta.
\end{align*}
The proof is done.
\end{proof}

\begin{lemma}\label{2604171624}
Let $N\ge 1$, $s\in (0, 1)$ and $p>1$. Assume that $\mathcal{U}$ is an open set such that $\overline{\mathcal{U}} \Subset \mathbb{R}^N_+$, and $K\Subset \mathbb{R}_+^N\backslash \overline{\mathcal{U}}$. Then
\[
\int_{K}\frac{1}{|x-y|^{N+sp}} {\rm d}y\le {\rm dist}(\mathcal{U}, K)^{-N-sp}|K|\quad \mbox{for all }x\in \mathcal{U},
\]
and 
\[
\underset{x\in \mathcal{U}, y\in K}{\sup}\frac{|x-y|}{|x-y^*|}<1,
\]
where ${\rm dist}(\mathcal{U}, K)=\inf{\{|x-y|: x\in \mathcal{U}, y\in K\}}$.
\end{lemma}
 
 The next result is a variant of \cite[Lemma 2.8]{IMS_RMI2016}, see also \cite[Lemma 3.3]{J_ANS2018}.
\begin{lemma}\label{2605211341}
Let $N\ge 1$, $s\in (0, 1)$ and $p>1$. Suppose that $\Omega \subset \mathbb{R}^N$ is a bounded open set, and $u\in \widetilde{W}^{s,p}(\Omega)$ is a weak solution of
\[
(-\Delta)_p^s u= g\quad \mbox{in } \Omega,
\]
for some $g\in L^1_{loc}(\Omega)\cap W_0^{s, p}(\Omega)'$. Let $v\in L^1_{loc}(\mathbb{R}^N)$  satisfy
\[
{\rm dist}({\rm supp}(v), \Omega)>0, \quad \int_{\Omega^c}\frac{|v(x)|^{p-1}}{(1+|x|)^{N+sp}} {\rm d}x<\infty.
\]
Define
\[
h(x)=2\int_{{\rm supp} (v)}\frac{J_{u+v}(x, y)-J_u(x, y)}{|x-y|^{N+sp}} {\rm d}y, \quad \mbox{for a.e. }x\in \Omega.
\]
Then $u+v\in \widetilde{W}^{s,p}(\Omega)$ and it solves weakly $(-\Delta)_p^s (u+v)= g+h$ in $\Omega$.
\end{lemma}

\section{Weak comparison principle}\label{S2}
In this section, we are devoted to the proof of Theorem \ref{2508070852}. We assume that $U\subset \mathbb{R}^N_+$ is a bounded open set. For $u$ defined over $U\cup\mathcal{T}(U)$, we denote $w=u-u^*$, which becomes an anti-symmetric function with respect to $x_N$.
\begin{lemma}\label{2602101903}
Let $N\ge 1$, $s\in (0, 1)$ and $p>1$. Assume that $U\subset \mathbb{R}^N_+$ is a bounded open set, $u\in \widetilde{W}^{s,p}(U)\cap \widetilde{W}^{s,p}(\mathcal{T}(U))$ and $w=u-u^*\ge 0$ a.e. in $\mathbb{R}^N_+\backslash U$. Then we have $w_-{\mathbbm 1}_{U}\in W_0^{s, p}(U)$.
\end{lemma}
\begin{proof}
By Lemma \ref{2605261121}, it suffices to verify $w_-{\mathbbm 1}_{U}\in \widetilde{W}^{s,p}(U)$. 
Set 
\[
D_1:=\{x\in U: w_-(x)>0\},\quad D_2:=\{y\in \mathbb{R}^N_-: w_-(y)>0\}.
\]
Suppose that $D_1\neq\varnothing$ and $D_2\neq\varnothing$. Since $w$ is anti-symmetric, for any $(x, y)\in D_1\times D_2$, there holds
\[
\begin{aligned}
\frac{|{w_-(x){\mathbbm 1}_{U}}(x)-{w_-(y){\mathbbm 1}_{U}}(y)|^p}{|x-y|^{N+sp}}=&\frac{|{w_-(x)}|^p}{|x-y|^{N+sp}}= \frac{|w_-(x)-w_-(y^*)|^p}{|x-y|^{N+sp}}\le \frac{|w_-(x)-w_-(y^*)|^p}{|x-y^*|^{N+sp}}.
\end{aligned}
\]
In the last inequality above, we used $|x-y|\ge |x-y^*|$. Hence
\begin{equation}\label{2412312226}
\begin{aligned}
\int_{D_1\times D_2}\frac{|{w_-(x){\mathbbm 1}_{U}}(x)-{w_-(y){\mathbbm 1}_{U}}(y)|^p}{|x-y|^{N+sp}}{\rm d}x{\rm d}y&\le \int_{D_1\times \mathcal{T}(D_2)}\frac{|{w_-(x)}-{w_-(y)}|^p}{|x-y|^{N+sp}}{\rm d}x{\rm d}y\\
&\le \int_{D_1\times \mathbb{R}^N}\frac{|{w_-(x)}-{w_-(y)}|^p}{|x-y|^{N+sp}}{\rm d}x{\rm d}y.
\end{aligned}
\end{equation}
If $D_1=\varnothing$ or $D_2=\varnothing$, \eqref{2412312226} will still hold. Due to the assumption $w\ge 0$ in $\mathbb{R}^N_+\backslash U$,  it is clear that
\begin{equation}\label{2412312227}
\begin{aligned}
\int_{D_1\times D_2^c}\frac{|{w_-(x){\mathbbm 1}_{U}}(x)-{w_-(y){\mathbbm 1}_{U}}(y)|^p}{|x-y|^{N+sp}}{\rm d}x{\rm d}y&=  \int_{D_1\times D_2^c}\frac{|{w_-(x)}-{w_-(y)}|^p}{|x-y|^{N+sp}}{\rm d}x{\rm d}y\\
&\le \int_{D_1\times \mathbb{R}^N}\frac{|{w_-(x)}-{w_-(y)}|^p}{|x-y|^{N+sp}}{\rm d}x{\rm d}y.
\end{aligned}
\end{equation}
On the other hand, we have
\begin{equation}\label{2605241515}
\begin{aligned}
\int_{(U\backslash D_1)\times \mathbb{R}^N}\frac{|{w_-(x){\mathbbm 1}_{U}}(x)-{w_-(y){\mathbbm 1}_{U}}(y)|^p}{|x-y|^{N+sp}}{\rm d}x{\rm d}y&=\int_{(U\backslash D_1)\times \mathbb{R}^N}\frac{|{w_-(y){\mathbbm 1}_{U}}(y)|^p}{|x-y|^{N+sp}}{\rm d}x{\rm d}y\\
&\le \int_{(U\backslash D_1)\times \mathbb{R}^N}\frac{|{w_-(y)}|^p}{|x-y|^{N+sp}}{\rm d}x{\rm d}y\\
&= \int_{(U\backslash D_1)\times \mathbb{R}^N}\frac{|{w_-(x)-w_-(y)}|^p}{|x-y|^{N+sp}}{\rm d}x{\rm d}y.\\
\end{aligned}
\end{equation}
By \eqref{2412312226}-\eqref{2605241515}, we get
\begin{equation}\label{2605241527}
\int_{U\times \mathbb{R}^N}\frac{|{w_-(x){\mathbbm 1}_{U}}(x)-{w_-(y){\mathbbm 1}_{U}}(y)|^p}{|x-y|^{N+sp}}{\rm d}x{\rm d}y\le 2\int_{U\times \mathbb{R}^N}\frac{|{w_-(x)}-{w_-(y)}|^p}{|x-y|^{N+sp}}{\rm d}x{\rm d}y.
\end{equation}
Moreover, $w=u-u^*\in \widetilde{W}^{s,p}(U)$ yields $w_-\in \widetilde{W}^{s,p}(U)$, then the right hand side of \eqref{2605241527} is finite, so  we have $ w_-{\mathbbm 1}_{U}\in \widetilde{W}^{s,p}(U)$. Thus the proof is completed.
\end{proof}
Next, we divide  $U\times U$ into four disjoint parts
 \begin{equation}\label{2603171604}
 \begin{aligned}
 A_1:=\{(x, y)\in U\times U: w(x)\ge 0,\; w(y)\ge 0\},\\
 A_2:=\{(x, y)\in U\times U: w(x)\ge 0,\; w(y)< 0\},\\
 A_3:=\{(x, y)\in U\times U: w(x)< 0,\; w(y)\ge 0\},\\
 A_4:=\{(x, y)\in U\times U: w(x)< 0,\; w(y)< 0\}.
 \end{aligned}
 \end{equation}
We claim 
\begin{lemma}\label{2602021059}
Let $N\ge 1$, $s\in (0, 1)$, $1<p<2$, and $u\in \widetilde{W}^{s,p}(U)\cap\widetilde{W}^{s,p}(\mathcal{T}(U))\cap  L^{\infty}(\mathbb{R}^N)$ where $U\subset  \mathbb{R}_+^N$ is a bounded open set.  If $w=u-u^*\ge 0$ a.e. in $\mathbb{R}^N_+\backslash U$, there holds 
\begin{equation}\label{2602010021}
\begin{aligned}
&\quad \int_{\mathbb{R}^{N}\times \mathbb{R}^{N}}\frac{(J_{u}(x, y)-J_{u^*}(x, y))\left(w_-(x){\mathbbm 1}_{U}(x)-w_-(y){\mathbbm 1}_{U}(y)\right)}{|x-y|^{N+sp}}{\rm d}x{\rm d}y\\
&\le-2^{p-1}(p-1)\|u\|^{p-2}_{L^\infty(\mathbb{R}^N)}\Big[\int_{A_3}\frac{w_-^2(x)}{|x-y|^{N+sp}}{\rm d}x{\rm d}y+\frac12\int_{A_4}\frac{\big|w_-(x)-w_-(y)\big|^2}{|x-y|^{N+sp}}{\rm d}x{\rm d}y\\
&\quad +\int_{A_4}\frac{w_-^2(x)}{|x-y^*|^{N+sp}}{\rm d}x{\rm d}y+\int_{U\times (\mathbb{R}^N_+\backslash U)}\frac{w_-^2(x)}{|x-y|^{N+sp}}{\rm d}x{\rm d}y\Big].\\
\end{aligned}
\end{equation}
\end{lemma}

\begin{proof}
If $\|u\|_{L^{\infty}(\mathbb{R}^N)}=0$, \eqref{2602010021} is trivial. Hence we assume $\|u\|_{L^{\infty}(\mathbb{R}^N)}> 0$. From Lemma \ref{2602101903}, it follows that $\varphi=w_-{\mathbbm 1}_{U} \in W_0^{s,p}(U)$, then the left hand side  of \eqref{2602010021} is well defined that we separate as follows.
\begin{align*}
I&:=\int_{\mathbb{R}^{N}\times \mathbb{R}^{N}}\frac{(J_{u}(x, y)-J_{u^*}(x, y))(\varphi(x)-\varphi(y))}{|x-y|^{N+sp}}{\rm d}x{\rm d}y\\
&=\int_{U\times U}\frac{(J_{u}(x, y)-J_{u^*}(x, y))(w_-(x)-w_-(y))}{|x-y|^{N+sp}}{\rm d}x{\rm d}y\\
& \quad +2\int_{U\times U^c}\frac{(J_{u}(x, y)-J_{u^*}(x, y))w_-(x)}{|x-y|^{N+sp}}{\rm d}x{\rm d}y\\
&=\int_{U\times U}\frac{(J_{u}(x, y)-J_{u^*}(x, y))(w_-(x)-w_-(y))}{|x-y|^{N+sp}}{\rm d}x{\rm d}y\\
&\quad +2\int_{U\times (\mathbb{R}^N_+\backslash U)}\frac{(J_{u}(x, y)-J_{u^*}(x, y))w_-(x)}{|x-y|^{N+sp}}{\rm d}x{\rm d}y\\
& \quad +2\int_{U\times \mathbb{R}^N_+}\frac{(J_{u}(x, y^*)-J_{u^*}(x, y^*))w_-(x)}{|x-y^*|^{N+sp}}{\rm d}x{\rm d}y.
\end{align*}
There holds then
\begin{align*}
I & =\int_{U\times U}\frac{(J_{u}(x, y)-J_{u^*}(x, y))(w_-(x)-w_-(y))}{|x-y|^{N+sp}}{\rm d}x{\rm d}y\\
& \quad +2\int_{U\times (\mathbb{R}^N_+\backslash U)}\frac{(J_{u}(x, y)-J_{u^*}(x, y))w_-(x)}{|x-y|^{N+sp}}{\rm d}x{\rm d}y\\
&\quad +2\int_{U\times U}\frac{(J_{u}(x, y^*)-J_{u^*}(x, y^*))w_-(x)}{|x-y^*|^{N+sp}}{\rm d}x{\rm d}y\\
& \quad +2\int_{U\times (\mathbb{R}^N_+\backslash U)}\frac{(J_{u}(x, y^*)-J_{u^*}(x, y^*))w_-(x)}{|x-y^*|^{N+sp}}{\rm d}x{\rm d}y\\
&~~=:  I_1+I_2+I_3+I_4.
\end{align*}
Using the region decomposition in \eqref{2603171604}, we get
\begin{equation}\label{2412201400}
\begin{aligned}
I_1& =-\int_{A_2}\frac{(J_{u}(x, y)-J_{u^*}(x, y))w_-(y)}{|x-y|^{N+sp}}{\rm d}x{\rm d}y+\int_{A_3}\frac{(J_{u}(x, y)-J_{u^*}(x, y))w_-(x)}{|x-y|^{N+sp}}{\rm d}x{\rm d}y\\
&\quad +\int_{A_4}\frac{(J_{u}(x, y)-J_{u^*}(x, y))\big(w_-(x)-w_-(y)\big)}{|x-y|^{N+sp}}{\rm d}x{\rm d}y\\
& = 2\int_{A_3}\frac{(J_{u}(x, y)-J_{u^*}(x, y))w_-(x)}{|x-y|^{N+sp}}{\rm d}x{\rm d}y\\
&\quad +\int_{A_4}\frac{(J_{u}(x, y)-J_{u^*}(x, y))\big(w_-(x)-w_-(y)\big)}{|x-y|^{N+sp}}{\rm d}x{\rm d}y\\
& =: I_{11}+I_{12},
\end{aligned}
\end{equation}
and
\begin{equation}\label{24122014011}
\begin{aligned}
I_3&=2\int_{A_3}\frac{(J_{u}(x, y^*)-J_{u^*}(x, y^*))w_-(x)}{|x-y^*|^{N+sp}}{\rm d}x{\rm d}y+2\int_{A_4}\frac{(J_{u}(x, y^*)-J_{u^*}(x, y^*))w_-(x)}{|x-y^*|^{N+sp}}{\rm d}x{\rm d}y\\
&=:I_{31}+I_{32}.
\end{aligned}
\end{equation}
In the sequel, we will estimate respectively the integrals $I_{11}+I_{31}$, $I_2+I_4$, $I_{12}$ and $I_{32}$.

\smallskip
\noindent {\bf Estimates on $I_{11}+I_{31}$.} We claim that  
\begin{align}\label{2606120109}
\begin{split}
& \quad \frac{(J_{u}(x, y^*)-J_{u^*}(x, y^*))w_-(x)}{|x-y^*|^{N+sp}}+\frac{(J_{u}(x, y)-J_{u^*}(x, y))w_-(x)}{|x-y|^{N+sp}}\\
& \le - \frac{2^{p-2}(p-1)\|u\|^{p-2}_{L^\infty(U\cup \mathcal{T}(U))}w_-^2(x)}{|x-y|^{N+sp}}, \qquad \mbox{for } \; (x, y)\in A_3.
\end{split}
\end{align}
We consider two subcases. 

\smallskip
\noindent{\sl Case 1: ${u^*(x)-u(y)}\ge {u(x)-u^*(y)}$.} Here we have 
\begin{equation}\label{2603171939-7}
u^*(x)-u^*(y)-u(x)+u(y)\ge w_-(x).
\end{equation}
Applying \eqref{2603171939-7} and Lemma \ref{Lo}, for such $(x, y)\in A_3$, there holds
\begin{align*}
&\quad \frac{(J_{u}(x, y^*)-J_{u^*}(x, y^*))w_-(x)}{|x-y^*|^{N+sp}}+\frac{(J_{u}(x, y)-J_{u^*}(x, y))w_-(x)}{|x-y|^{N+sp}}\\
& = \frac{(p-1)w_-(x)}{|x-y^*|^{N+sp}}\int^{u(x)-u^*(y)}_{u^*(x)-u(y)} |t|^{p-2} {\rm d}t- \frac{(p-1)w_-(x)}{|x-y|^{N+sp}}\int^{u^*(x)-u^*(y)}_{u(x)-u(y)} |t|^{p-2} {\rm d}t\\
&\le - \frac{(p-1)w_-(x)}{|x-y|^{N+sp}}\int^{u^*(x)-u^*(y)}_{u(x)-u(y)} |t|^{p-2} {\rm d}t\\  
&\le- \frac{(p-1)w_-^2(x)}{|x-y|^{N+sp}}\min\big\{|u^*(x)-u^*(y)|^{p-2}, |u(x)-u(y)|^{p-2}\big\}\\  
&\le- \frac{2^{p-2}(p-1)\|u\|^{p-2}_{L^\infty(U\cup \mathcal{T}(U))}w_-^2(x)}{|x-y|^{N+sp}}.
\end{align*}
For the above first equality, we used 
\begin{equation}\label{2603171939}
J_{u}(x, y^*)-J_{u^*}(x, y^*) = (p-1)\int^{u(x)-u^*(y)}_{u^*(x)-u(y)} |t|^{p-2} {\rm d}t,
\end{equation}
and 
\begin{equation}\label{2603171940}
J_{u}(x, y)-J_{u^*}(x, y) = -(p-1)\int_{u(x)-u(y)}^{u^*(x)-u^*(y)} |t|^{p-2} {\rm d}t.
\end{equation}

\noindent 
{\sl Case 2: ${u^*(x)-u(y)}\le {u(x)-u^*(y)}$.} As $(x, y)\in A_3$, there holds
\begin{equation}\label{2607301930}
u^*(x)-u^*(y)\ge u(x)-u^*(y)\ge u^*(x)-u(y)\ge u(x)-u(y).
\end{equation}
By \eqref{2603171939}-\eqref{2607301930} and Lemma \ref{Lo}, for such $(x, y)\in A_3$,
\begin{align*}
&\quad \frac{(J_{u}(x, y^*)-J_{u^*}(x, y^*))w_-(x)}{|x-y^*|^{N+sp}}+\frac{(J_{u}(x, y)-J_{u^*}(x, y))w_-(x)}{|x-y|^{N+sp}}\\
&\le\frac{(p-1)w_-(x)}{|x-y|^{N+sp}}\int^{u(x)-u^*(y)}_{u^*(x)-u(y)} |t|^{p-2} {\rm d}t- \frac{(p-1)w_-(x)}{|x-y|^{N+sp}}\int^{u^*(x)-u^*(y)}_{u(x)-u(y)} |t|^{p-2} {\rm d}t\\
&=\frac{(p-1)w_-(x)}{|x-y|^{N+sp}}\Big[\int^{u(x)-u^*(y)}_{u^*(x)-u(y)} |t|^{p-2} {\rm d}t- \int^{u^*(x)-u^*(y)}_{u(x)-u(y)} |t|^{p-2} {\rm d}t\Big]\\
&=-\frac{(p-1)w_-(x)}{|x-y|^{N+sp}}\Big[\int^{u^*(x)-u^*(y)}_{u(x)-u^*(y)} |t|^{p-2} {\rm d}t+ \int^{u^*(x)-u(y)}_{u(x)-u(y)} |t|^{p-2} {\rm d}t\Big]\\
&\le- \frac{2(p-1)w_-^2(x)}{|x-y|^{N+sp}}\min\big\{|u^*(x)-u^*(y)|^{p-2}, |u(x)-u(y)|^{p-2}\big\}\\
&\le- \frac{2^{p-1}(p-1)\|u\|^{p-2}_{L^\infty(U\cup \mathcal{T}(U))}w_-^2(x)}{|x-y|^{N+sp}}.
\end{align*}
Then the claim \eqref{2606120109} holds true, which yields 
\begin{equation}\label{2412201402}
I_{11}+I_{31}
\le-2^{p-1}(p-1)\|u\|^{p-2}_{L^\infty(U\cup \mathcal{T}(U))}\int_{A_3}\frac{w_-^2(x)}{|x-y|^{N+sp}}{\rm d}x{\rm d}y.
\end{equation}

\noindent
{\bf Estimates on $I_2+I_4$.} We claim that  
\begin{align}\label{2603251150}
\begin{split}
& \quad \frac{\big(J_{u}(x, y)-J_{u^*}(x, y)\big)w_-(x)}{|x-y|^{N+sp}}+\frac{\big(J_{u}(x, y^*)-J_{u^*}(x, y^*)\big)w_-(x)}{|x-y^*|^{N+sp}}\\
& \le -\frac{2^{p-2}(p-1)\|u\|^{p-2}_{L^\infty(\mathbb{R}^N)}w_-^2(x)}{|x-y|^{N+sp}}, \qquad \mbox{for }\; (x, y)\in U\times (\mathbb{R}^N_+\backslash U).
\end{split}
\end{align}
As above, we consider two different situations.

\smallskip
\noindent
{\sl Case 1: ${u^*(x)-u(y)}\ge {u(x)-u^*(y)}$.} 
By \eqref{2603171939}, \eqref{2603171940}, $w\ge 0$ in $\mathbb{R}^N_+\backslash U$ and Lemma \ref{Lo}, it holds that for $(x, y)\in U\times (\mathbb{R}^N_+\backslash U)$,  
\begin{align*}
&\quad \frac{\big(J_{u}(x, y)-J_{u^*}(x, y)\big)w_-(x)}{|x-y|^{N+sp}}+\frac{\big(J_{u}(x, y^*)-J_{u^*}(x, y^*)\big)w_-(x)}{|x-y^*|^{N+sp}}\\
&= \frac{(p-1)w_-(x)}{|x-y|^{N+sp}}\int_{u^*(x)-u^*(y)}^{u(x)-u(y)} |t|^{p-2} {\rm d}t- \frac{(p-1)w_-(x)}{|x-y^*|^{N+sp}}\int^{u^*(x)-u(y)}_{u(x)-u^*(y)} |t|^{p-2} {\rm d}t\\ 
&\le \frac{(p-1)w_-(x)}{|x-y|^{N+sp}}\int_{u^*(x)-u^*(y)}^{u(x)-u(y)} |t|^{p-2} {\rm d}t\\
&\le -\frac{(p-1)w_-^2(x)}{|x-y|^{N+sp}}\min\big\{|u^*(x)-u^*(y)|^{p-2}, |u(x)-u(y)|^{p-2}\}\\
&\le -\frac{2^{p-2}(p-1)\|u\|^{p-2}_{L^\infty(\mathbb{R}^N)}w_-^2(x)}{|x-y|^{N+sp}}.
\end{align*}

\noindent
{\sl Case 2: ${u^*(x)-u(y)}\le {u(x)-u^*(y)}$.} Thus \eqref{2607301930} still holds in $U\times (\mathbb{R}^N_+\backslash U)$.
Hence by \eqref{2603171939}-\eqref{2607301930}, $w\ge 0$ in $\mathbb{R}^N_+\backslash U$ and Lemma \ref{Lo}, for $(x, y)\in U\times (\mathbb{R}^N_+\backslash U)$, we have
\begin{align*}
&\quad \frac{\big(J_{u}(x, y)-J_{u^*}(x, y)\big)w_-(x)}{|x-y|^{N+sp}}+\frac{\big(J_{u}(x, y^*)-J_{u^*}(x, y^*)\big)w_-(x)}{|x-y^*|^{N+sp}}\\
&= \frac{(p-1)w_-(x)}{|x-y|^{N+sp}}\int_{u^*(x)-u^*(y)}^{u(x)-u(y)} |t|^{p-2} {\rm d}t+ \frac{(p-1)w_-(x)}{|x-y^*|^{N+sp}}\int_{u^*(x)-u(y)}^{u(x)-u^*(y)} |t|^{p-2} {\rm d}t\\ 
&\le -\frac{(p-1)w_-(x)}{|x-y|^{N+sp}}\int^{u^*(x)-u^*(y)}_{u(x)-u(y)} |t|^{p-2} {\rm d}t+\frac{(p-1)w_-(x)}{|x-y|^{N+sp}}\int_{u^*(x)-u(y)}^{u(x)-u^*(y)} |t|^{p-2} {\rm d}t\\
&=-\frac{(p-1)w_-(x)}{|x-y|^{N+sp}}\Big[\int^{u^*(x)-u^*(y)}_{u(x)-u^*(y)} |t|^{p-2} {\rm d}t+ \int_{u(x)-u(y)}^{u^*(x)-u(y)} |t|^{p-2} {\rm d}t\Big]\\
&\le -\frac{2(p-1)w_-^2(x)}{|x-y|^{N+sp}}\min\big\{|u^*(x)-u^*(y)|^{p-2}, |u(x)-u(y)|^{p-2}\big\}\\
&\le -\frac{2^{p-1}(p-1)\|u\|^{p-2}_{L^\infty(\mathbb{R}^N)}w_-^2(x)}{|x-y|^{N+sp}}.
\end{align*}
Then the claim \eqref{2603251150} is valid, which means that
\begin{equation}\label{2412201401}
I_2+ I_4\le -2^{p-1}(p-1)\|u\|^{p-2}_{L^\infty(\mathbb{R}^N)}\int_{U\times (\mathbb{R}^N_+\backslash U)}\frac{w_-^2(x)}{|x-y|^{N+sp}}{\rm d}x{\rm d}y.
\end{equation}

\noindent
{\bf Estimates on $I_{12}$.} As before, using Lemma \ref{Lo},  for $(x, y)\in A_4$,
\begin{align*}
&\quad -(J_{u}(x, y)-J_{u^*}(x, y))\big(w_-(x)-w_-(y)\big)\\
&=(p-1)\big(w_-(x)-w_-(y)\big)\int_{u(x)-u(y)}^{u^*(x)-u^*(y)}|t|^{p-2} {\rm d}t\\
&\ge (p-1)\min\big\{|u^*(x)-u^*(y)|^{p-2}, |u(x)-u(y)|^{p-2}\big\}|w_-(x)-w_-(y)|^2\\
&\ge 2^{p-2}(p-1)\|u\|_{L^{\infty}(U\cup \mathcal{T}(U))}^{p-2}|w_-(x)-w_-(y)|^2.
\end{align*}
We deduce then
\begin{equation}\label{2603251505}
\begin{aligned}
I_{12}\le -2^{p-2}(p-1) \|u\|_{L^{\infty}(U\cup \mathcal{T}(U))}^{p-2}\int_{A_4}\frac{\big|w_-(x)-w_-(y)\big|^2}{|x-y|^{N+sp}}{\rm d}x{\rm d}y.
\end{aligned}
\end{equation}

\smallskip
\noindent
{\bf Estimates on $I_{32}$.} By direct computations, 
\begin{equation}\label{2603150136}
\begin{aligned}
I_{32}&=2(p-1)\int_{A_4}\frac{w_-(x)}{|x-y^*|^{N+sp}}\int_{u^*(x)-u(y)}^{u(x)-u^*(y)} |t|^{p-2} {\rm d}t {\rm d}x{\rm d}y\\
&\le -2(p-1)\int_{A_4}\frac{w_-^2(x)\min\big\{|u(x)-u^*(y)|^{p-2}, |u^*(x)-u(y)|^{p-2}\big\}}{|x-y^*|^{N+sp}} {\rm d}x{\rm d}y\\
&\le -2^{p-1}(p-1)\|u\|^{p-2}_{L^\infty(U\cup \mathcal{T}(U))}\int_{A_4}\frac{w_-^2(x)}{|x-y^*|^{N+sp}} {\rm d}x{\rm d}y.\\
\end{aligned}
\end{equation}

From \eqref{2412201400},  \eqref{24122014011}, \eqref{2412201402}, \eqref{2412201401}, \eqref{2603251505} and \eqref{2603150136}, we conclude that
\begin{align*}
I&=I_{11}+I_{12}+I_2+I_{31}+I_{32}+I_4\\
& =(I_{11}+I_{31})+(I_2+I_4)+I_{12}+I_{32}\\
&\le-c_p\|u\|^{p-2}_{L^\infty(U\cup \mathcal{T}(U))}\times \Big[\int_{A_3}\frac{w_-^2(x)}{|x-y|^{N+sp}}{\rm d}x{\rm d}y+\frac12\int_{A_4}\frac{\big|w_-(x)-w_-(y)\big|^2}{|x-y|^{N+sp}}{\rm d}x{\rm d}y\\
&\quad  +\int_{A_4}\frac{w_-^2(x)}{|x-y^*|^{N+sp}}{\rm d}x{\rm d}y\Big] -c_p\|u\|^{p-2}_{L^\infty(\mathbb{R}^N)}\int_{U\times (\mathbb{R}^N_+\backslash U)}\frac{w_-^2(x)}{|x-y|^{N+sp}}{\rm d}x{\rm d}y,
\end{align*}
where $c_p=2^{p-1}(p-1)$, which implies readily \eqref{2602010021}. Thus the proof is completed.
\end{proof}

To establish Theorem \ref{2508070852}, we should estimate the right hand side of \eqref{2602010021}, which is the aim of the following Proposition. Notice that for $p \le 2$,
$$W_0^{s, p}(U)\cap L^{\infty}(U) \subset W_0^{\frac{sp}{2}, 2}(U).$$ Indeed, for any $v\in W_0^{s, p}(U)\cap L^{\infty}(U)$,
\[
\int_{\mathbb{R}^N}\int_{\mathbb{R}^N}\frac{|v(x)-v(y)|^2}{|x-y|^{N+sp}}{\rm d}x{\rm d}y\le 2^{2-p}\|v\|_{L^{\infty}(U)}^{2-p}\int_{\mathbb{R}^N}\int_{\mathbb{R}^N}\frac{|v(x)-v(y)|^p}{|x-y|^{N+sp}}{\rm d}x{\rm d}y.
\]
\begin{proposition} \label{2602021049}
Let $N\ge 1$, $s\in (0, 1)$, $1<p<2$, $U\subset  \mathbb{R}_+^N$ be a bounded open set and $$u\in \widetilde{W}^{s,p}(U)\cap \widetilde{W}^{s,p}(\mathcal{T}(U))\cap L^\infty(\mathbb{R}^N).$$  If  $w=u-u^*\ge 0$ a.e. in $\mathbb{R}^N_+\backslash U$, then we have
\begin{align}\label{2602021036}
\begin{split}
& \quad -\int_{\mathbb{R}^{2N}}\frac{(J_{u}(x, y)-J_{u^*}(x, y))\left(w_-(x){\mathbbm 1}_{U}(x)-w_-(y){\mathbbm 1}_{U}(y)\right)}{|x-y|^{N+sp}}{\rm d}x{\rm d}y\\
& \ge C \|u\|_{L^{\infty}(\mathbb{R}^N)}^{p-2}\|w_-{\mathbbm 1}_{U}\|_{W_0^{\frac{sp}{2},2}(U)}^2,
\end{split}
\end{align}
where $C$ depends only on $p$.
\end{proposition}

\begin{proof} Assume again $\|u\|_{L^{\infty}(\mathbb{R}^N)}>0$. Clearly,
\begin{equation}\label{2603251506}
\int_{A_3}\frac{w_-^2(x)}{|x-y|^{N+sp}}{\rm d}x{\rm d}y=\int_{A_2}\frac{w_-^2(y)}{|x-y|^{N+sp}}{\rm d}x{\rm d}y,
\end{equation}
where $A_2$, $A_3$ are given in \eqref{2603171604}. Using  \eqref{2603251506}, we get
\begin{equation}\label{2603271227}
\begin{aligned}
\int_{A_3}\frac{w_-^2(x)}{|x-y|^{N+sp}}{\rm d}x{\rm d}y+&\frac12\int_{A_4}\frac{\big|w_-(x)-w_-(y)\big|^2}{|x-y|^{N+sp}}{\rm d}x{\rm d}y= \frac12\int_{U\times U}\frac{\big|w_-(x)-w_-(y)\big|^2}{|x-y|^{N+sp}}{\rm d}x{\rm d}y.\\
\end{aligned}
\end{equation}
Next, let $D=\mathbb{R}^N\backslash (U\cup \mathcal{T}(U))$. Since  $|x-y|\le |x-y^*|$, there holds
\begin{align*}
\int_{U\times (\mathbb{R}^N_+\backslash U)}\frac{w_-^2(x)}{|x-y|^{N+sp}}{\rm d}x{\rm d}y
&\ge \frac12\int_{U\times (\mathbb{R}^N_+\backslash U)}\frac{w_-^2(x)}{|x-y|^{N+sp}}{\rm d}x{\rm d}y+\frac12\int_{U\times (\mathbb{R}^N_+\backslash U)}\frac{w_-^2(x)}{|x-y^*|^{N+sp}}{\rm d}x{\rm d}y\\
&= \frac12\int_{U\times (\mathbb{R}^N_+\backslash U)}\frac{w_-^2(x)}{|x-y|^{N+sp}}{\rm d}x{\rm d}y+\frac12\int_{U\times \mathcal{T}(\mathbb{R}^N_+\backslash U)}\frac{w_-^2(x)}{|x-y|^{N+sp}}{\rm d}x{\rm d}y\\
&=\frac12\int_{U\times D}\frac{|w_-(x){\mathbbm 1}_{U}(x)-w_-(y){\mathbbm 1}_{U}(y)|^2}{|x-y|^{N+sp}}{\rm d}x{\rm d}y\\
&=\frac14\int_{(U\times D)\cup (D\times U)}\frac{|w_-(x){\mathbbm 1}_{U}(x)-w_-(y){\mathbbm 1}_{U}(y)|^2}{|x-y|^{N+sp}}{\rm d}x{\rm d}y\\
&=\frac14\int_{(\mathbb R^N\times D)\cup (D\times\mathbb R^N)}\frac{|w_-(x){\mathbbm 1}_{U}(x)-w_-(y){\mathbbm 1}_{U}(y)|^2}{|x-y|^{N+sp}}{\rm d}x{\rm d}y.
\end{align*}
Therefore, as $D^c = U\cup \mathcal{T}(U)$, we get
\begin{align}\label{2603271228}
\begin{split}
& \quad \int_{U\times (\mathbb{R}^N_+\backslash U)}\frac{w_-^2(x)}{|x-y|^{N+sp}}{\rm d}x{\rm d}y\\
&\ge \frac14\int_{\mathbb{R}^{2N}\backslash (U\times U)}\frac{|w_-(x){\mathbbm 1}_{U}(x)-w_-(y){\mathbbm 1}_{U}(y)|^2}{|x-y|^{N+sp}}{\rm d}x{\rm d}y\\
& \quad -\frac14\int_{\mathcal{T}(U)\times U}\frac{w_-^2(y)}{|x-y|^{N+sp}}{\rm d}x{\rm d}y - \frac14\int_{U \times\mathcal{T}(U)}\frac{w_-^2(x)}{|x-y|^{N+sp}}{\rm d}x{\rm d}y\\
&= \frac14\int_{\mathbb{R}^{2N}\backslash (U\times U)}\frac{|w_-(x){\mathbbm 1}_{U}(x)-w_-(y){\mathbbm 1}_{U}(y)|^2}{|x-y|^{N+sp}}{\rm d}x{\rm d}y-\frac{Q_\omega}{2},
\end{split}
\end{align}
where
\begin{equation}\label{2603271233}
Q_\omega=\int_{U\times \mathcal{T}(U)}\frac{w_-^2(x)}{|x-y|^{N+sp}} {\rm d}x{\rm d}y.
 \end{equation}
Take
 \begin{equation*}
 \begin{aligned}
 E_1:=\{(x, y)\in U\times \mathcal{T}(U): w(x)< 0,\; w(y)\ge 0\},\\
 E_2:=\{(x, y)\in U\times \mathcal{T}(U): w(x)< 0,\; w(y)< 0\}.
 \end{aligned}
 \end{equation*}
Since $w(y)=-w(y^*)$, we obtain
\begin{align*}
Q_1:=\int_{E_1}\frac{w_-^2(x)}{|x-y|^{N+sp}} {\rm d}x{\rm d}y= \int_{A_4}\frac{w_-^2(x)}{|x-y^*|^{N+sp}} {\rm d}x{\rm d}y,
 \end{align*}
and
\begin{align*}
Q_2:=\int_{E_2}\frac{w_-^2(x)}{|x-y|^{N+sp}} {\rm d}x{\rm d}y=\int_{A_3}\frac{w_-^2(x)}{|x-y^*|^{N+sp}} {\rm d}x{\rm d}y\le \int_{A_3}\frac{w_-^2(x)}{|x-y|^{N+sp}} {\rm d}x{\rm d}y.
 \end{align*}
Applying \eqref{2603271227}, it follows that 
\begin{equation}\label{2603271237}
\begin{split}
Q_\omega = Q_1+Q_2 & \le \int_{A_3}\frac{w_-^2(x)}{|x-y|^{N+sp}} {\rm d}x{\rm d}y+\int_{A_4}\frac{w_-^2(x)}{|x-y^*|^{N+sp}} {\rm d}x{\rm d}y\\
& \le \frac12\int_{U\times U}\frac{\big|w_-(x)-w_-(y)\big|^2}{|x-y|^{N+sp}}{\rm d}x{\rm d}y + \int_{A_4}\frac{w_-^2(x)}{|x-y^*|^{N+sp}} {\rm d}x{\rm d}y.
\end{split}
 \end{equation}
On the other hand, from \eqref{2608311117}-\eqref{2603311321}, we know that 
\begin{equation}\label{2608311123}
\|w_-{\mathbbm 1}_{U}\|_{W_0^{\frac{sp}{2},2}(U)}^2=\int_{\mathbb{R}^N}\int_{\mathbb{R}^N}\frac{|w_-(x){\mathbbm 1}_{U}(x)-w_-(x){\mathbbm 1}_{U}(y)|^2}{|x-y|^{N+sp}}{\rm d}x{\rm d}y.
\end{equation}
Combining \eqref{2603271227}-\eqref{2608311123} and Lemma \ref{2602021059}, we conclude \eqref{2602021036}.
\end{proof}
\begin{remark}
One can check step by step that Lemma \ref{2602021059} and Proposition \ref{2602021049} still hold if $p=2$. Moreover, as $p=2$, the assumption $u\in L^{\infty}(\mathbb{R}^N)$ can be removed.
\end{remark}

Now we are able to complete the proof of Theorem \ref{2508070852}.

\begin{proof}[\bf Proof of Theorem \ref{2508070852} completed.] Suppose $\|u\|_{L^\infty(\mathbb{R}^N)}> 0$, we choose $\varphi=-w_-{\mathbbm 1}_{U}$ as a test function. Since $u$ satisfies \eqref{2412171452}, there holds
\begin{equation}\label{24171502}
\int_{\mathbb{R}^{2N}}\frac{(J_{u}(x, y)-J_{u^*}(x, y))(\varphi(x)-\varphi(y))}{|x-y|^{N+sp}}{\rm d}x{\rm d}y\le \int_{U}V(x)w_-^2 {\rm d}x.
\end{equation}
Combining with Proposition \ref{2602021049}, H\"older inequality and Lemma \ref{2603291635}, we deduce that
\begin{align*}
\|V_+\|_{L^{\infty}(U)}\|w_-\|_{L^2(U)}^2 \ge\int_{U}V(x)w_-^2 {\rm d}x &\ge C_1\|u\|_{L^{\infty}(\mathbb{R}^N)}^{p-2}\|w_-{\mathbbm 1}_{U}\|_{W_0^{\frac{sp}{2},2}(U)}^2\\
&\ge C_1 \lambda_{1, \frac{sp}{2}, 2}(U)\|u\|_{L^{\infty}(\mathbb{R}^N)}^{p-2}\|w_-\|_{L^2(U)}^2\\
&\ge C|U|^{-\frac{sp}{N}}\|u\|_{L^{\infty}(\mathbb{R}^N)}^{p-2}\|w_-\|_{L^2(U)}^2,
\end{align*}
where $C_1$, $C$ depend only on $N, p$ and $s$. Thus $\|w_-\|_{L^2(U)}=0$ if \eqref{2603291622} holds, 
which means that $u\ge u^*$ a.e. in $U$.
\end{proof}

\section{Strong comparison principle}\label{2606091353}

In this section, we will prove Theorem \ref{2605211200}. The main idea is constructing  barrier functions and applying the weak comparison principle in Theorem \ref{2508070852}.

\begin{lemma}\label{2604131506}
Let $N\ge 1$, $s\in (0, 1)$, $1<p<2$, $\mathcal{U}$ be an open set satisfying $\overline{\mathcal U} \Subset \mathbb{R}_+^N$, and $K\Subset \mathbb{R}^N_+\backslash \overline{\mathcal{U}}$. Then there exists  $C=C(\mathcal{U}, K, N, s, p)>0$  such that given any  $\delta>0$ and $v\in \widetilde{W}^{s,p}(\mathcal{U})\cap \widetilde{W}^{s,p}(\mathcal{T}(\mathcal{U}))\cap L^{\infty}(\mathbb{R}^N)$, if
\begin{equation}\label{2605111909}
|v(x)-v^*(x)|\le \Big[\frac{\beta (p-1)(2\|v\|_{L^{\infty}(K\cup \mathcal{U})}+\delta)^{p-2} \, {\rm dist}(\mathcal{U}, \mathcal{T}(K))^{N+sp}\delta}{2^{3-p}{\rm diam}(\mathcal{U}, K)^{N+sp}}\Big]^{\frac{1}{p-1}} \quad \mbox{in }\; \mathcal{U},
\end{equation}
with ${\rm diam}(\mathcal{U}, K)=\sup\{|x-y|: x\in \mathcal{U}, y\in K\}$ and
\begin{equation}\label{2605111536}
\beta:=1-\underset{x\in \mathcal{U}, y\in K}{\sup}\frac{|x-y|^{N+sp}}{|x-y^*|^{N+sp}}\in (0, 1),
\end{equation}
one has then, in weak sense,
\begin{align*}
& \quad (-\Delta)_p^s (v-\delta {\mathbbm 1}_{K})-(-\Delta)_p^s (v^*-\delta {\mathbbm 1}_{\mathcal{T}(K)})\\
&\ge  (-\Delta)_p^s v- (-\Delta)_p^s v^* + C(\|v\|_{L^{\infty}(K\cup \mathcal{U})}+\delta)^{p-2}\delta\quad \mbox{in }\; \mathcal{U}.
\end{align*}
\end{lemma}
\begin{proof} 
By Lemma \ref{2605211341},
\[
(-\Delta)_p^s (v-\delta {\mathbbm 1}_{K})= (-\Delta)_p^s v+H_\delta\quad \mbox{weakly in } \mathcal{U},
\]
where
\begin{equation}\label{2605111533}
H_\delta=2\int_{K}\frac{\mathcal{P}(v(x)-v(y)+\delta {\mathbbm 1}_{K}(y))-\mathcal{P}(v(x)-v(y))}{|x-y|^{N+sp}} {\rm d}y,
\end{equation}
recall that $\mathcal{P}(t)=|t|^{p-2}t$. Similarly, we have
\[
(-\Delta)_p^s (v^*-\delta {\mathbbm 1}_{\mathcal{T}(K)})= (-\Delta)_p^s v^*+G_\delta\quad \mbox{weakly in } \mathcal{U},
\]
where
\begin{equation}\label{2605111534}
\begin{aligned}
G_\delta&=2\int_{\mathcal{T}(K)}\frac{\mathcal{P}(v^*(x)-v^*(y)+\delta {\mathbbm 1}_{\mathcal{T}(K)}(y))-\mathcal{P}(v^*(x)-v^*(y))}{|x-y|^{N+sp}} {\rm d}y\\
&=2\int_{K}\frac{\mathcal{P}(v^*(x)-v(y)+\delta {\mathbbm 1}_{K}(y))-\mathcal{P}(v^*(x)-v(y))}{|x-y^*|^{N+sp}} {\rm d}y.\\
\end{aligned}
\end{equation}
Now, we denote
\begin{equation}\label{26051119000}
\alpha =\frac{\beta (p-1)(2\|v\|_{L^{\infty}(K\cup \mathcal{U})}+\delta)^{p-2} \, {\rm dist}(\mathcal{U}, \mathcal{T}(K))^{N+sp}\delta}{2{\rm diam}(\mathcal{U}, K)^{N+sp}}.
\end{equation}
By assumption \eqref{2605111909}, Lemmas \ref{2605111444}-\ref{2604171624} and $|v(x)-v^*(x)|\le (2^{p-2}\alpha)^\frac{1}{p-1}$ in $\mathcal{U}$, so that
\begin{align*}
G_\delta & \le 2\int_{K}\frac{\mathcal{P}(v(x)-v(y)+\delta {\mathbbm 1}_{K}(y))-\mathcal{P}(v(x)-v(y))+\alpha}{|x-y^*|^{N+sp}} {\rm d}y\\
& \le 2(1-\beta)\int_{K}\frac{\mathcal{P}(v(x)-v(y)+\delta {\mathbbm 1}_{K}(y))-\mathcal{P}(v(x)-v(y))}{|x-y|^{N+sp}} {\rm d}y + 2\alpha{\rm dist}(\mathcal{U}, \mathcal{T}(K))^{-N-sp}|K|\\
& = (1-\beta)H_\delta + 2\alpha{\rm dist}(\mathcal{U}, \mathcal{T}(K))^{-N-sp}|K|.
\end{align*}
Consequently, using Lemma \ref{2605111444}, there holds
\[
\begin{aligned}
H_\delta-G_\delta &\ge \beta H_\delta - 2\alpha{\rm dist}(\mathcal{U}, \mathcal{T}(K))^{-N-sp}|K|\\
&\ge 2\beta (p-1)\big(2\|v\|_{L^{\infty}(K\cup \mathcal{U})}+\delta\big)^{p-2}\delta \, {\rm diam}(\mathcal{U}, K)^{-N-sp}|K|- 2\alpha{\rm dist}(\mathcal{U}, \mathcal{T}(K))^{-N-sp}|K|\\
&=\beta (p-1)\big(2\|v\|_{L^{\infty}(K\cup \mathcal{U})}+\delta\big)^{p-2}\delta \, {\rm diam}(\mathcal{U}, K)^{-N-sp}|K|\\
&= C(2\|v\|_{L^{\infty}(K\cup \mathcal{U})}+\delta)^{p-2}\delta,
\end{aligned}
\]
where $C=\beta(p-1){\rm diam}(\mathcal{U}, K)^{-N-sp}|K|$. Therefore, weakly in $\mathcal{U}$, we have
\[
\begin{aligned}
(-\Delta)_p^s (v-\delta {\mathbbm 1}_{K})-(-\Delta)_p^s (v^*-\delta {\mathbbm 1}_{\mathcal{T}(K)})&=(-\Delta)_p^s v-(-\Delta)_p^s v^*+H_\delta-G_\delta\\
&\ge (-\Delta)_p^s v-(-\Delta)_p^s v^*+ C(\|v\|_{L^{\infty}(K\cup \mathcal{U})}+\delta)^{p-2}\delta.
\end{aligned}
\]
We finish the proof.
\end{proof}

The following result was proved in \cite[Lemma 3.5(i)]{J_ANS2018}.

\begin{lemma}\label{2605121130}
Given $N\ge 1$, $s\in (0, \frac12)$, $\frac{1}{1-s}<p<2$, an open set $U\subset \mathbb{R}^N$, and $g\in C_c^2(U)$. Then there is $C>0$ depending only on $N$, $p$ and $s$, such that for any $u\in \widetilde{W}^{s, p}(U)$ and $t\in [-1, 1]$, in weak sense
\[
\Big|(-\Delta)_p^s (u-t g)-(-\Delta)_p^s u\Big|\le C\|g\|_{C^2(\mathbb{R}^N)}^{p-1}|t|^{p-1}\quad \mbox{in } {\rm supp} (g),
\]
where $\|g\|_{C^2(\mathbb{R}^N)}=\Sigma_{|\alpha|\le 2}\|\partial^\alpha g\|_{L^{\infty}(\mathbb{R}^N)}$.
\end{lemma}

Now, we are in position to prove the strong comparison for continuous weak solutions to \eqref{2412171452}. 
\begin{proof}[\bf Proof of Theorem \ref{2605211200}]
Assume that $u\not\equiv u^*$ in $\mathbb{R}^N$, and assume by contradiction that $u(x_0)=u^*(x_0)$ for some $x_0\in U$.
There exists a compact set $K\Subset \mathbb{R}^N_+$ with positive measure such that $x_0\not\in K$ and
\[
\gamma_K:=\underset{{K}}{\min}\, (u-u^*)>0.
\]
Fix a sequence $0<r_k<\frac14 {\rm dist}(x_0, U^c\cup K)$ such that $\lim_{k\to\infty} r_k = 0$. As $u$ is continuous, we have
$$\underset{B_{r_k}(x_0)}{\sup}\, (u-u^*)\to 0\quad \mbox{as } k\to\infty.$$
Fix $0<\delta\le \frac12\min\big\{\gamma_K, 1\big\}$. Apply Lemma \ref{2604131506} with $\mathcal{U}=B_{2r_k}(x_0)$ for $r_k$ sufficiently small such that \eqref{2605111909} holds, this is possible since the right hand side of \eqref{2605111909} tends to a positive constant when $r_k \to 0$. Then there exist $C_k>0$ such that
\begin{equation}\label{2605131922}
(-\Delta)_p^s z_\delta-(-\Delta)_p^s z_{\delta}^*\ge C_k\delta +V(x)(u-u^*)\quad \mbox{in}\,\, B_{2r_k}(x_0),
\end{equation}
where $z_\delta=u-\delta{\mathbbm 1}_K$. Seeing the proof of Lemma \ref{2604131506}, we can suppose that $C_k\ge C>0$.

\smallskip
It is clear that 
\[
z_\delta-z_\delta^*=u-u^*-\delta{\mathbbm 1}_K+\delta{\mathbbm 1}_{\mathcal{T}(K)}\ge 0 \quad \mbox{in}\,\, \mathbb{R}_+^N.
\]
Next, let $g\in C_c^{2}(B_{2})$ be a cut-off function such that ${\mathbbm 1}_{B_1}\le g\le {\mathbbm 1}_{B_2}$.
 Take
\[
g_k(x)=g\Big(\frac{x-x_0}{r_k}\Big).
\]
For any $1\ge\varepsilon>0$, set $h_{\delta, \varepsilon, k}=z_\delta-\varepsilon g_k$. 
It follows from Lemma \ref{2605121130} that
\begin{equation}\label{2605131924}
(-\Delta)_p^s h_{\delta, \varepsilon, k}\ge  (-\Delta)_p^s z_\delta- C'{r_k}^{2(1-p)}\varepsilon^{p-1}\quad  \mbox{in }\; {\rm supp} (g_k)\subset B_{2r_k}(x_0).
\end{equation}
We will provide an upper bound for $(-\Delta)_p^s h_{\delta, \varepsilon,k}^*$ in ${\rm supp} (g_k)$. Since ${\rm dist}({\rm supp} (g_k), \mathcal{T}(K))>0$, by virtue of  Lemma \ref{2605211341}, there holds
\begin{equation}\label{2605131920}
(-\Delta)_p^s h_{\delta, \varepsilon, k}^*=  (-\Delta)_p^s z_\delta^*+ H_{\delta, \varepsilon, k}\quad  \mbox{in }\; {\rm supp} (g_k),
\end{equation}
where
\begin{equation*}
\begin{aligned}
H_{\delta, \varepsilon, k}(x)&=2\int_{{\rm supp} (g_k^*)}\frac{\mathcal{P}(z^*_\delta(x)-z^*_\delta(y)+\varepsilon g_k^*(y))-\mathcal{P}(z^*_\delta(x)-z^*_\delta(y))}{|x-y|^{N+sp}} {\rm d}y.
\end{aligned}
\end{equation*}
Using Lemma \ref{2605111444}, for $x\in {\rm supp} (g_k)$, we have
\begin{equation}\label{2605131921}
\begin{aligned}
H_{\delta, \varepsilon, k}(x)&\le 2^{3-p}\varepsilon^{p-1}\int_{{\rm supp}(g_k^*)}\frac{1}{|x-y|^{N+sp}} {\rm d}y\\
&\le 2^{3-p}\varepsilon^{p-1}{\rm dist}\big({\rm supp}(g_k), {\rm supp} (g_k^*)\big)^{-N-sp}|{\rm supp}(g_k^*)|\\
&\le 2^{N+sp+3-p}|x_0-x_0^*|^{-N-sp}|B_{2r_k}(x_0)|\varepsilon^{p-1}\\
&=C''|x_0-x_0^*|^{-N-sp}r_k^N\varepsilon^{p-1},
\end{aligned}
\end{equation}
where $C''=C''(N, p, s)>0$.
It follows from \eqref{2605131920} and \eqref{2605131921}  that
\begin{equation}\label{2605131923}
(-\Delta)_p^s h_{\delta, \varepsilon,k}^*\le (-\Delta)_p^s z_\delta^*+ C''|x_0-x_0^*|^{-N-sp}r_k^N\varepsilon^{p-1}\quad\mbox{weakly in }\; {\rm supp}(g_k).
\end{equation}
Finally, we get from \eqref{2605131922}, \eqref{2605131924} and \eqref{2605131923} that in ${\rm supp} (g_k)$, 
\begin{align*}
    (-\Delta)_p^s h_{\delta, \varepsilon, k}-(-\Delta)_p^s h_{\delta, \varepsilon, k}^* &\ge C\delta-C'{r_k}^{2(1-p)}\varepsilon^{p-1} -C''|x_0-x_0^*|^{-N-sp}r_k^N\varepsilon^{p-1}\\
& \quad +V(x)(u-u^*).
\end{align*}

Let now $\varepsilon_k = o(r_k^2) >0$ be small enough such that
\[
C\delta-C'{r_k}^{2(1-p)}\varepsilon_k^{p-1}-C''|x_0-x_0^*|^{-N-sp}r_k^N\varepsilon_k^{p-1}\ge \frac{C\delta}2 \ge  \varepsilon_k\|V\|_{L^\infty({\rm supp} (g_k))}.
\]
In ${\rm supp}(g_k)$, there holds
\[
\begin{aligned}
(-\Delta)_p^s h_{\delta, \varepsilon_k, k}-(-\Delta)_p^s h_{\delta, \varepsilon_k, k}^* \ge \frac{C\delta}2 +V(x)(u-u^*)
&\ge V(x)(u-u^*-\varepsilon_k g_k)\\
&=V(x)(h_{\delta, \varepsilon_k, k}-h_{\delta, \varepsilon_k, k}^*).
\end{aligned}
\]
Using Theorem \ref{2508070852} with small $r_k$, we arrive at 
$$h_{\delta, \varepsilon_k, k}-h_{\delta, \varepsilon_k, k}^*=u-u^*-\varepsilon_k \ge 0 \quad \mbox{in } B_{r_k}(x_0),$$
which contradicts the assumption  $u(x_0)=u^*(x_0)$. The hypothesis was wrong and the proof is completed.
\end{proof}

\section{Applications}\label{2606091358}
As an example of applications, by the moving-plane method, we show some symmetry result, that is, Theorem \ref{2602111215} and Corollary \ref{2603271652}. The main steps to carry out the moving-plane procedure are very similar to that in \cite[Theorem 1.5]{D_AIHP1998} (see also \cite[Section 4]{JW_AMPA2016} and \cite[Theorem 1.3]{BN_BSBM1991}), hence we omit a lot of generic details, and just mention some special points.

\begin{proof}[\bf Proof of Theorem \ref{2602111215}] Assume first that $u\not\equiv 0$. Similar to \cite[Theorem 1.5]{D_AIHP1998} and \cite[Section 4]{JW_AMPA2016}, we make use of the weak comparison principle, Theorem \ref{2508070852} to get the monotonicity of solutions near the boundary, and start the moving-plane. Furthermore, we shall prove that the moving-plane will reach the position $x_N= 0$. This step  essentially originated from the proof of \cite[Theorem 1.3]{BN_BSBM1991}, where the crucial argument is to apply strong and weak maximum principles in a bounded domain. For the nonlocal or quasi-linear setting, we can follow the same idea, and we refer to the proofs of \cite[Theorem 1.5]{D_AIHP1998} and \cite[Section 4]{JW_AMPA2016}. In particular, the strong comparison principle in Theorem \ref{2605211200} plays the role of \cite[Theorem 1.4]{D_AIHP1998} and \cite[Proposition 3.6]{JW_AMPA2016} respectively.
\end{proof}

\begin{proof}[\bf Proof of Corollary \ref{2603271652}]
Since $s\in (0, \frac12)$ and $p\in (\frac{1}{1-s}, 2]$,  we have $sp<1\le N$. With the assumption \eqref{2608241033}, using \cite[Theorem 3.3]{CMS_JFA2018}, any weak solution to \eqref{main} is bounded and continuous. So any non-negative weak solution $u$ will satisfy the requirement of Theorem \ref{2602111215}. Hence the solution $u$ is radially symmetric and strictly decreasing with respect to the radius.
\end{proof}

\bigskip
\noindent
{\bf Acknowledgements.} D.Y. is partially supported by NSFC (No.~12271164) and Science and Technology Commission of Shanghai Municipality (No.~22DZ2229014). W.Z. is partially supported by NSFC (No.~12601212).

\medskip
\noindent{\bf Data availability.}
Data sharing is not applicable to this work as no new data were created or analyzed in the study.

\medskip
\noindent{\bf Conflict of interest statement.}
There is no conflict of interest relevant to
this article.


\begin{thebibliography}{777}

\bibitem{A_AMPA1961} A.D. Alexandrov, A characteristic property of the spheres, {\it Ann. Math. Pura Appl.}, 58 (1962), 303-315.


\bibitem{BMS_JAM2018} B. Barrios, L. Montoro and B. Sciunzi, On the moving plane method for nonlocal problems in bounded domains,
{\it J. Anal. Math.}, 135 (2018), 37-57.

\bibitem{BN_BSBM1991} H. Berestycki and L. Nirenberg, On the method of moving planes and the sliding method,
{\it Bol. Soc. Bras. Mat.}, 22 (1991), 1-37.

\bibitem{BRS_arXiv2026} A. Biswas, S. Roy and  A. Sen, Strong comparison principle and symmetry results for the fractional $p$-Laplacian, (2026), arXiv:2606.08559.

\bibitem{BT_APDE2025} A. Biswas and E. Topp, Lipschitz regularity of fractional $p$-Laplacian, {\it Ann. PDE}, 11 (2025), paper no. 27.


\bibitem{BDLMS_CVPDE2025} V. B\"ogelein, F. Duzaar, N. Liao, G. Molica  Bisci and R. Servadei, Gradient regularity for $(s, p)$-harmonic functions, {\it Calc. Var. Partial Differential Equations}, 64 (2025), paper no. 253.

\bibitem{BDLM_JMPA2025} V. B\"ogelein, F. Duzaar, N. Liao and K. Moring, Gradient estimates for the fractional $p$-Poisson equation, {\it J. Math. Pures Appl.}, 204 (2025), paper no. 103764.

\bibitem{BL_AM2017} L. Brasco and E. Lindgren,  Higher Sobolev regularity for the fractional $p$-Laplace equation in the superquadratic case, {\it Adv. Math.}, {304} (2017), 300-354.

\bibitem{BLS_AM2018} L. Brasco, E. Lindgren and A. Schikorra,   {Higher H\"older regularity for the fractional $p$-Laplacian in the superquadratic case}, {\it Adv. Math.}, {338} (2018), 782-846.

\bibitem{B_PIAS2000} F. Brock, Continuous rearrangement and symmetry of solutions of elliptic problems, {\it Proc. Indian Acad. Sci.
(Math. Sci.)}, 110 (2000), 157-204.

\bibitem{CL_AM2018} W. Chen and C. Li, Maximum principles for the fractional $p$-Laplacian and symmetry of solutions, {\it Adv. Math.}, 335 (2018), 735-758.

\bibitem{CLL_AM2017} W. Chen, C. Li and Y. Li, A direct method of moving planes for the fractional
Laplacian, {\it Adv. Math.}, 308 (2017), 404-437.

\bibitem{CMS_JFA2018} W. Chen, S. Mosconi and M. Squassina, Nonlocal problems with critical Hardy nonlinearity,
 {\it J. Funct. Anal.}, 275 (2018), 3065-3114.
 
\bibitem{D_AIHP1998} L. Damascelli, Comparison theorems for some quasilinear degenerate elliptic operators and applications to symmetry and monotonicity results, {\it Ann. Inst. H. Poincar\'e. Analyse non lin\'eaire}, 15 (1998), 493-516.
 
\bibitem{DP_ASNSPCS1998} L. Damascelli and F. Pacella, Monotonicity and symmetry of solutions of $p$-Laplace equations, $1<p<2$; via the moving plane method, {\it Ann. Scuola Norm. Sup. Pisa Cl. Sci.}, 26 (1998), 689-707.


\bibitem{DS_JDE2004} L. Damascelli and B. Sciunzi,  Regularity, monotonicity and symmetry of positive solutions of $m$-Laplace equations, {\it J. Differential Equations}, 206 (2004), 483-515.

\bibitem{DS_CVPDE2006} L. Damascelli and B. Sciunzi, Harnack inequalities, maximum and comparison principles, and regularity of positive solutions of $m$-Laplace equations, {\it Calc. Var. Partial Differential Equations}, 25 (2006), 139-159.

\bibitem{dLN_JMPA1982} D. de Figueiredo, P. Lions and R. Nussbaum, A priori estimates and existence of positive solutions of semilinear elliptic equations, {\it J. Math. Pures Appl.}, 61 (1982),  41-63.

\bibitem{DN_APDE2025} L. Diening and S. Nowak,  Calder\'on-Zygmund estimates for the fractional $p$-Laplacian, {\it Ann. PDE}, 11 (2025), paper no. 6.

\bibitem{DPV_BSM12} 
\newblock E. Di Nezza, G. Palatucci and E. Valdinoci, 
\newblock  {Hitchhiker's guide to the fractional Sobolev spaces}, 
\newblock \emph{Bull. Sci. Math.}, {136} (2012), 521-573.

\bibitem{GL_MA2024} P. Garain and E. Lindgren, 
Higher H\"older regularity for the fractional $p$-Laplace equation in the subquadratic case,
{\it Math. Ann.}, 390 (2024),  5753-5792.

\bibitem {GYZ_CAM2002} Y. Ge,  D. Ye and F. Zhou,  Comparison, symmetry and monotonicity results for some degenerate elliptic operators in Carnot-Carath\'eodory spaces, {\it Chinese Ann. Math. Ser. B}, 23 (2002), 361-372.


\bibitem{GNN_CMP1979} B. Gidas, W. Ni and L. Nirenberg,  Symmetry and related properties via the maximum principle, {\it Comm. Math. Phys.}, 68 (1979), 209-243.

\bibitem{GJS_arXiv2025} D. Giovagnoli, D. Jesus and L. Silvestre, $C^{1+\alpha}$-regularity for fractional $p$-harmonic equations, (2025), arXiv:2509.26565.

\bibitem{GV_NA1989} M. Guedda and L. V\'eron, Quasilinear elliptic equations involving critical Sobolev exponents, {\it Nonlinear Anal.},
13 (1989), 879-902.

\bibitem{GP_JAM2025} Y. Guo and S. Peng, Maximum principles and Direct methods for fractional Hardy operator and applications, {\it J. Anal. Math.}, 156 (2025), 171-211.

\bibitem{IMP_MN23} A. Iannizzotto, S. Mosconi and N. Papageorgiou, On the logistic equation for the fractional p-Laplacian. {\it Math. Nachr.} 296(4) (2023), 1451-1468.

\bibitem{IMS_RMI2016} A. Iannizzotto, S. Mosconi and M. Squassina,  Global H\"{o}lder regularity for the fractional $p$-Laplacian, {\it Rev. Mat. Iberoam.},  32 (2016), 1353-1392.

\bibitem{IM_RMI1989} T. Iwaniec and J. Manfredi, Regularity of $p$-harmonic functions on the plane, {\it Rev. Mat. Iberoam.}, 5 (1989), 1-19.

\bibitem{J_ANS2018} S. Jarohs, Strong comparison principle for the
fractional $p$-Laplacian and applications to
starshaped rings, {\it Adv. Nonlinear Stud.},  18 (2018), 691-704.

\bibitem{JW_AMPA2016} S. Jarohs and T. Weth,  
Symmetry via antisymmetric maximum principles in nonlocal problems of variable order,
{\it Ann. Mat. Pura Appl.}, 195 (2016),  273-291.

\bibitem{KS_N1990} S. Kichenassamy and J. Smoller, On the existence of radial solutions of quasi-linear elliptic equations, {\it Nonlinearity},  3 (1990), 677-694.

\bibitem{KKL_JMPA2019} J. Korvenp\"a\"a, T. Kuusi and E. Lindgren, Equivalence of solutions to fractional $p$-Laplace type equations, {\it J. Math. Pures Appl.}, 132 (2019), 1-26.

\bibitem{LS_ArXiv2026} H. Lee and K. Song, Self-improving properties for the fractional $p$-Laplacian via nonlinear commutators, (2026), arXiv:2606.26610.

\bibitem{M_PAMS88} J. Manfredi, $p$-harmonic functions in the plane. Proc. Amer. Math. Soc. 103(2) (1988), 473-479.

\bibitem{RS_CVPDE2014} X. Ros-Oton and J. Serra, The extremal solution for the fractional Laplacian, {\it Calc. Var. Partial Differential Equations}, 50 (2014), 723-750.

\bibitem{S_ARMA1971} J. Serrin, A symmetry problem in potential theory, {\it Arch. Rational Mech. Anal.}, 43 (1971), 304-318.

\bibitem{T_CPDE1983} P. Tolksdorf, On the Dirichlet problem for quasilinear equations in domain with conical boundary points, {\it Comm. Partial Differential Equations}, 8 (1983), 773-817.

\end{thebibliography}
\end{document}